\pdfoutput=1
\documentclass[11pt]{article}

\usepackage[utf8]{inputenc}
\usepackage{graphicx}
\usepackage[svgnames]{xcolor}
\usepackage{array,fullpage,enumitem,microtype}
\usepackage[leqno]{amsmath}
\usepackage{amssymb,amsthm,mathtools,bm,stmaryrd,tikz-cd}
\usepackage[unicode,bookmarksnumbered,bookmarksopen,psdextra,colorlinks=true,allcolors=DarkBlue]{hyperref}
\usepackage{bookmark}
\usepackage{tocbibind}
\usepackage[capitalize,nosort]{cleveref}

\usepackage{mymacros}

\newcommand*\op{\mathrm{op}}

\newcommand*\Sub{\mathrm{Sub}}
\makeatletter
\newcommand*\aqf{\mathrm{\newmcodes@-aqf}}
\makeatother

\let\defn\textbf

\mathlig{=>}{\Rightarrow}
\mathlig{<=}{\Leftarrow}
\mathlig{<=>}{\Leftrightarrow}
\mathlig{|-}{\vdash}
\mathlig{|=}{\models}

\begin{document}

\title{A universal characterization of standard Borel spaces}
\author{Ruiyuan Chen}
\date{}
\maketitle

\begin{abstract}
We prove that the category $\!{SBor}$ of standard Borel spaces is the (bi-)initial object in the 2-category of countably complete Boolean (countably) extensive categories.  This means that $\!{SBor}$ is the universal category admitting some familiar algebraic operations of countable arity (e.g., countable products, unions) obeying some simple compatibility conditions (e.g., products distribute over disjoint unions).  More generally, for any infinite regular cardinal $\kappa$, the dual of the category $\kappa\!{Bool}_\kappa$ of $\kappa$-presented $\kappa$-complete Boolean algebras is (bi-)initial in the 2-category of $\kappa$-complete Boolean ($\kappa$-)extensive categories.
\let\thefootnote=\relax
\footnotetext{2020 \emph{Mathematics Subject Classification}:
    03G30, 
    03E15, 
    06E75, 
    18C10. 
}
\footnotetext{\emph{Key words and phrases}: standard Borel space, $\kappa$-complete Boolean algebra, extensive category, categorical logic.}
\end{abstract}

\section{Introduction}

A \defn{standard Borel space} is a measurable space (i.e., set equipped with $\sigma$-algebra of subsets) which is isomorphic to a Borel subspace of Cantor space $2^\#N$.  Standard Borel spaces and \defn{Borel maps} (i.e., preimages of Borel sets are Borel) are ubiquitous in descriptive set theory as a basic model of ``definable sets'' and ``definable functions'' between them; see \cite{Kcdst}.  The notion of ``definability'' here is a coarse one where, roughly speaking, all countable information is considered definable.  As a result, standard Borel spaces are closed under familiar set operations of countable arity, e.g., countable products, countable unions, injective Borel images.

In this paper, we give a purely ``synthetic'' characterization of the category $\!{SBor}$ of standard Borel spaces and Borel maps: we prove that $\!{SBor}$ is the free category generated by some familiar set operations (e.g., those above) subject to some obvious compatibility conditions between them (e.g., products distribute over unions).  In other words, every standard Borel space or Borel map can be regarded as a formal expression built from these operations (e.g., $\bigsqcup_{n \le \aleph_0} 2^n$); and more significantly, every true statement about the spaces or maps denoted by such expressions (e.g., a given map is an isomorphism) can be deduced in a purely formal manner from the compatibility conditions.

The operations on $\!{SBor}$ take the form of certain well-behaved limits and colimits, which we now briefly describe; see \cref{sec:cat} for the precise definitions.
Recall that \defn{limits} in a category $\!C$ give abstract versions of such set operations as products (including nullary, i.e., the 1-element set), preimages of subsets, inverse limits, and the equality binary relation.
Limits that exist in $\!C$ are automatically ``fully compatible'' with themselves.%
\footnote{They obey all compatibility conditions that hold in the category $\!{Set}$ of sets, by the Yoneda embedding.}
We call $\!C$ \defn{complete}, or more generally \defn{$\kappa$-complete} for a regular cardinal $\kappa$, if $\!C$ has all ($\kappa$-ary) limits.
\defn{Colimits}, dual to limits, are abstract versions of such set operations as disjoint unions (given by \defn{coproducts}), images, and quotients.
There are many well-studied compatibility conditions between limits and colimits, often known as ``exactness conditions'', based on conditions that hold in familiar categories such as $\!{Set}$; see e.g., \cite{GL}.
For a category $\!C$ with finite limits and coproducts, the condition known as \defn{extensivity} \cite{CLW}, based on the behavior of disjoint unions in $\!{Set}$, says that a map $Y -> \bigsqcup_{i \in I} X_i \in \!C$ into a finite coproduct is the same thing as a partition $Y = \bigsqcup_{i \in I} Y_i$ together with maps $Y_i -> X_i$ for each $i$.  More generally, $\!C$ is \defn{$\kappa$-extensive} if it has $\kappa$-ary coproducts obeying the analogous condition.
Finally, we say that a extensive category $\!C$ is \defn{Boolean} if the subobject poset $\Sub_\!C(X)$ of each object $X \in \!C$ is a Boolean algebra.

\begin{theorem}
\label{thm:intro-sbor-init}
The category $\!{SBor}$ of standard Borel spaces and Borel maps is the initial object in the 2-category of countably complete Boolean (countably) extensive categories.
\end{theorem}

Here, the morphisms $\@F : \!C -> \!D$ between two countably complete Boolean countably extensive categories are the functors preserving all of these operations, i.e., the \defn{countably continuous, countably extensive} functors.  As with most freeness statements about categories equipped with algebraic operations, we necessarily must consider the \defn{2-category} of countably complete Boolean countably extensive categories, countably continuous countably extensive functors between them, \emph{and natural transformations between such functors}.  This is because algebraic operations on a category (e.g., limits) are almost never uniquely defined, but only defined up to canonical isomorphism; thus, the usual universal property of a free algebra must be weakened to take natural isomorphisms between functors into account.  Accordingly, the notion of \defn{initial object} in \cref{thm:intro-sbor-init} refers to what is known as a \defn{bicolimit} in 2-category theory (see \cite{BKP}): for any other countably complete Boolean countably extensive category $\!C$, the category of countably continuous countably extensive functors $\!{SBor} -> \!C$ is (not necessarily isomorphic but) equivalent to the one-object one-morphism category.  In particular, there is a unique such functor $\!{SBor} -> \!C$, up to unique natural isomorphism between any two such functors.  Note that this characterizes $\!{SBor}$ up to equivalence.

The parenthetical in \cref{thm:intro-sbor-init} is because countable limits, finite extensivity, and Booleanness easily imply countable extensivity (see \cref{thm:klimbext-kext}).

We would like to stress that \cref{thm:intro-sbor-init} characterizes $\!{SBor}$ \emph{as an abstract category}, i.e., without an \emph{a priori} given underlying set functor $\!{SBor} -> \!{Set}$.  Based on the inductive definition of Borel sets, it is fairly obvious that $\!{SBor}$ as a (conservative) subcategory of $\!{Set}$ is generated by the specified operations.%
\footnote{A \defn{conservative} subcategory is one closed under existing inverses.
From the nullary product $1$, we get the disjoint union $2 = 1 \sqcup 1$, hence the product $2^\#N$; pulling back $1 `-> 2$ along projections $2^\#N -> 2$ gives the subbasic clopen sets in $2^\#N$, hence by Booleanness and countable completeness, all Borel subspaces of $2^\#N$.  Since a map between standard Borel spaces is Borel iff its graph is \cite[14.12]{Kcdst}, by conservativity applied to projections from graphs, we get all Borel maps.}
Thus, the force of \cref{thm:intro-sbor-init} is that any true statement in $\!{SBor}$ only mentioning these operations may be deduced purely formally from the compatibility axioms.

There is a classical result much along the lines of \cref{thm:intro-sbor-init}, and forming an important part of its proof: the \defn{Loomis--Sikorski representation theorem} says (in one equivalent formulation) that the Borel $\sigma$-algebra $\@B(2^\#N)$ of Cantor space is freely generated, as an \emph{abstract} Boolean $\sigma$-algebra (countably complete Boolean algebra), by the subbasic clopen sets in $2^\#N$; see \cref{sec:loomis-sikorski}.  An easy consequence is that we have a dual equivalence of categories
\begin{align*}
\!{SBor} &\cong \sigma\!{Bool}_\sigma^\op \\
X &|-> \@B(X) \\
(f : X -> Y) &|-> (f^{-1} : \@B(Y) -> \@B(X))
\end{align*}
with the category $\sigma\!{Bool}_\sigma$ of countably presented Boolean $\sigma$-algebras.  In fact, this equivalence, which we will refer to as \defn{Loomis--Sikorski duality}, is the only role of classical descriptive set theory in \cref{thm:intro-sbor-init}, which is a consequence of Loomis--Sikorski duality together with

\begin{theorem}[\cref{thm:kboolk-init}]
\label{thm:intro-kboolk-init}
For any infinite regular cardinal $\kappa$, the dual $\kappa\!{Bool}_\kappa^\op$ of the category $\kappa\!{Bool}_\kappa$ of $\kappa$-presented $\kappa$-complete Boolean algebras is the initial object in the 2-category of $\kappa$-complete Boolean ($\kappa$-)extensive categories.
\end{theorem}

When $\kappa = \omega_1$, this reduces to \cref{thm:intro-sbor-init} by Loomis--Sikorski duality.  Also note that when $\kappa = \omega$, we recover the well-known fact that the category of finite sets (dual to $\omega\!{Bool}_\omega$, the category of finite Boolean algebras) is the initial finitely complete extensive category (see e.g., \cite[A1.4.7]{Jeleph}).

We now briefly discuss the proof of \cref{thm:intro-kboolk-init}.  By the Lusin--Suslin theorem \cite[15.1]{Kcdst}, Borel injections between standard Borel spaces have Borel images, whence subobjects in $\!{SBor}$ are Borel subspaces; analogously, by a result of Lagrange \cite{Lag} (see \cref{thm:kboolk-quotient}), subobjects of $A \in \kappa\!{Bool}_\kappa^\op$ are in canonical bijection with elements $a \in A$.  Thus, in order to exploit the universal property of $\kappa$-presented algebras in the proof of \cref{thm:intro-kboolk-init}, it is convenient to reformulate the notion of $\kappa$-complete Boolean $\kappa$-extensive category in terms of ``objects and subobjects'' instead of ``objects and morphisms''.  This is achieved by the standard notion of \defn{hyperdoctrine} from categorical logic, due to Lawvere \cite{Law}.
For a category $\!C$ with enough compatible limits and colimits, we may equivalently describe $\!C$ as a collection of objects together with, for each tuple $\vec{A}$ of objects in $\!C$, a lattice $\Sub_\!C(\vec{A})$ of ``$\vec{A}$-ary predicates'', thought of as the subobject lattice of $\prod \vec{A}$, equipped with ``variable substitution'' and ``quantification'' maps between them as well as distinguished ``equality predicates'' $({=_A}) \in \Sub_\!C(A, A)$ for each $A$.
From this data, called the \defn{subobject hyperdoctrine} of $\!C$, $\!C$ may be reconstructed by taking morphisms to be binary predicates which are ``function graphs'' (expressed internally).  
See \cite{Law}, \cite{Jac} for details.

Rather than work directly with hyperdoctrines, we prefer to work with their presentations, i.e., first-order theories.  Given a multi-sorted relational first-order theory $\@T$ in a language $\@L$, a hyperdoctrine may be constructed whose objects are the sorts and whose predicates are the $\@T$-provable equivalence classes of formulas.  Constructing a category from this hyperdoctrine as above yields the \defn{syntactic category} $\ang{\@L \mid \@T}$ of the theory $\@T$, which is the category ``presented'' by $(\@L, \@T)$ in terms of objects (the sorts), subobjects (relation symbols in $\@L$), and relations between subobjects (axioms of $\@T$); see \cite[Ch.~8]{MR}, \cite[D1.3]{Jeleph}.
We will give in \cref{sec:authy} in some detail the one-step construction (bypassing hyperdoctrines) of syntactic categories for theories presenting $\kappa$-complete Boolean $\kappa$-extensive categories.  One reason for our level of detail is that the fragment of first-order logic whose theories present such categories does not appear to have been well-studied, and involves some technical complications: the fragment consists of formulas in the $\kappa$-ary infinitary first-order logic $\@L_{\kappa\kappa}$ (with $\kappa$-ary quantifiers and $\bigwedge, \bigvee$) which are ``\defn{almost quantifier-free}'' in that $\forall$ does not appear, and every $\exists$ that appears is already provably unique modulo the theory.  (This ``provably unique $\exists$'' is familiar in categorical logic from \defn{finite-limit logic} or \defn{Cartesian logic}, the fragment of $\@L_{\omega\omega}$ used to present finitely complete categories; see \cite[D1.3.4]{Jeleph}.)

The method of proof of \cref{thm:intro-kboolk-init} will be to show that $\kappa\!{Bool}_\kappa^\op$ is presented by a ``trivial'' theory axiomatizing structure automatically present in every $\kappa$-complete Boolean $\kappa$-extensive category.  That is, we show that $\kappa\!{Bool}_\kappa^\op$ is equivalent to the syntactic category of such a theory.  This will be easy for the part of the syntactic category consisting of quantifier-free formulas.  So the main difficulty will be to show that the ``provably unique $\exists$'' quantifier can be provably eliminated from formulas.  In terms of $\!{SBor}$, this amounts to showing that an injective Borel image can be ``algebraically witnessed''; thus the key to the proof will be an analysis of the algebraic content of the proof of the Lusin--Suslin theorem and its $\kappa$-ary analog (see \cref{rmk:lusin-suslin,lm:qee}).

\subsection{Connections with other work}

Abstract characterizations of various categories of topological spaces are known.  It is folklore that the category of Stone spaces is the pro-completion of the category of finite sets (see \cite[VI~2.3]{Jstone}).  In \cite{MRg}, a characterization is given of the category of compact Hausdorff spaces.


Recently, there have been efforts to formulate a \emph{synthetic descriptive set theory}; see Pauly--de~Brecht \cite{PdB}.
The goal of such a program is to uniformly recover key notions and results from several known versions of descriptive set theory (classical, effective, generalized to other kinds of spaces, e.g., \cite{Sel}) in an abstract manner from the relevant categories of spaces, thereby providing an explanation for the informal analogies between these theories.  \Cref{thm:intro-sbor-init} shows that in the boldface, purely Borel context (without any topology), we have a complete answer as to how much categorical structure and axioms are required to fully recover the classical theory in the form of $\!{SBor}$.  Moreover, the recovery here takes the satisfying form of a universal property uniquely characterizing $\!{SBor}$ among all possible ``models'' of these axioms.

The Loomis--Sikorski duality $\!{SBor} \cong \sigma\!{Bool}_\sigma^\op$ is part of a family of analogies (see \cite{Isb}, \cite{Hec}, \cite{Cqpol}, \cite[arXiv version]{Cscc}) between descriptive set theory and \emph{pointless topology} or \emph{locale theory} (see \cite{Jstone}, \cite[C1.1--2]{Jeleph}) suggesting that elementary
classical descriptive set theory is in many ways the ``countably presented fragment'' of locale theory.
Adopting this point of view, one can define a \defn{standard $\kappa$-Borel locale} $X$ to formally mean the same thing as a $\kappa$-presented $\kappa$-complete Boolean algebra $\@B_\kappa(X)$;
a \defn{$\kappa$-Borel map} $f : X -> Y$ between such $X, Y$ is a $\kappa$-Boolean homomorphism $f^* : \@B_\kappa(Y) -> \@B_\kappa(X)$.  Thus, standard $\kappa$-Borel locales are to standard Borel spaces as $\kappa$-copresented $\kappa$-locales (see \cite{Mad}) are to de~Brecht's \emph{quasi-Polish spaces} (see \cite{deB}, \cite{Hec}).  \Cref{thm:intro-kboolk-init} then becomes the literal $\kappa$-ary generalization of \cref{thm:intro-sbor-init}.  The proof of \cref{thm:intro-kboolk-init} is largely informed by this point of view as well, e.g., to show that $\kappa\!{Bool}_\kappa^\op$ is indeed $\kappa$-complete $\kappa$-extensive (see \cref{thm:kboolk-prod,thm:kboolk-ext}), or the aforementioned result of Lagrange \cite{Lag} which is a $\kappa$-ary version of the Lusin separation theorem playing an analogous role in the proof of $\kappa$-ary Lusin--Suslin (see \cref{rmk:lusin-suslin}).
See \cite{Cloc} for a systematic development of these ideas.
To keep this paper as concrete as possible, we will not explicitly refer to standard $\kappa$-Borel locales in what follows.

\subsection{Contents of paper}

This paper should hopefully be readable with basic category-theoretic background, although some familiarity with categorical logic would be helpful.

In \cref{sec:cat}, we review some (2-)category theory, including extensive categories and variants.

In \cref{sec:kbool}, we review and introduce needed background on $\kappa$-Boolean algebras, including the results of Lagrange \cite{Lag} as well as the fact that $\kappa\!{Bool}_\kappa^\op$ is $\kappa$-complete Boolean $\kappa$-extensive.

In \cref{sec:loomis-sikorski}, we review Loomis--Sikorski duality, which we formulate in terms of a Stone-type dual adjunction (in the sense of e.g., \cite[VI~\S4]{Jstone}) between measurable spaces and Boolean $\sigma$-algebras.  We include a self-contained proof.

Finally, in \cref{sec:authy}, we develop the ``almost quantifier-free'' fragment of $\@L_{\kappa\kappa}$, show that its theories present (via the syntactic category construction) $\kappa$-complete Boolean $\kappa$-(sub)extensive categories, and use such a theory to prove \cref{thm:intro-kboolk-init}.

\paragraph*{Acknowledgments}

I would like to thank Matthew de~Brecht for some stimulating discussions that partly inspired this work, as well as the referee for some helpful corrections and suggestions.

\section{Categories}
\label{sec:cat}

We assume familiarity with basic notions of category theory, including e.g., limits, colimits, adjoint functors; see e.g., \cite{Mac}, \cite{Bor}.  Categories will be denoted with sans-serif symbols like $\!C$.  Hom-sets will be denoted $\!C(X, Y)$ for $X, Y \in \!C$; identity morphisms will be denoted $1_X : X -> X$.  The category of sets will be denoted $\!{Set}$.  The \defn{terminal category} has a single object and single (identity) morphism.

Fix throughout the paper an infinite regular cardinal $\kappa$; the case $\kappa = \omega_1$ will be typical.  By \defn{$\kappa$-ary}, we mean of size $<\kappa$.  Let $\!{Set}_\kappa \setle \!{Set}$ be the full subcategory of $\kappa$-ary sets.  Also fix some set $\#U$ of size $\ge \kappa$ (e.g., $\#U := \kappa$), for convenience assuming $\#N \setle \#U$, and let $\!{Set}'_\kappa \setle \!{Set}_\kappa$ be the full subcategory of $\kappa$-ary subsets of $\#U$; thus $\!{Set}'_\kappa \subseteq \!{Set}_\kappa$ is a small full subcategory containing a set of each cardinality $<\kappa$, hence equivalent to $\!{Set}_\kappa$.

Recall that a \defn{subobject} of an object $X \in \!C$ of a category $\!C$ is an equivalence class of monomorphisms $f : A `-> X \in \!C$ with respect to the preorder
\begin{equation*}
(f : A `-> X) \setle (g : B `-> X) \coloniff \exists h : A -> B\, (f = g \circ h)
\end{equation*}
(such an $h$ is necessarily unique).  As is common, we will often abuse terminology regarding subobjects: we also refer to single monomorphisms as subobjects, identified with their equivalence classes; we sometimes write a subobject $f : A `-> X$ (i.e., its equivalence class) simply as $A \setle X$; and we generally use familiar notations for subsets with the obvious meanings (e.g., $\cup, \cap$).  Let
\begin{equation*}
\Sub(X) = \Sub_\!C(X) := \{\text{subobjects of $X$ in $\!C$}\},
\end{equation*}
partially ordered by (the partial order induced on equivalence classes by) $\setle$.

A category $\!C$ is \defn{$\kappa$-complete} if it has all $\kappa$-ary limits; a functor $\@F : \!C -> \!D$ between two $\kappa$-complete categories is \defn{$\kappa$-continuous} if it preserves all $\kappa$-ary limits.  In a $\kappa$-complete category, each subobject poset $\Sub(X)$ has $\kappa$-ary meets given by (wide) pullback; we denote meets of $A_i \setle X$ by either $\bigcap_i A_i$, or $\bigwedge_i A_i$, or (as pullbacks) $\prod_X (A_i)_i$.  For each $f : X -> Y$, the pullback map
\begin{align*}
\Sub(f) = f^* : \Sub(Y) -> \Sub(X)
\end{align*}
between subobject posets preserves $\kappa$-ary meets; we thus get a functor
\begin{align*}
\Sub : \!C^\op --> \!{\kappa{\bigwedge}Lat}
\end{align*}
to the category $\!{\kappa{\bigwedge}Lat}$ of $\kappa$-complete meet semilattices.
A $\kappa$-continuous functor $\@F : \!C -> \!D$ between $\kappa$-complete categories preserves monomorphisms, hence induces a map between subobject posets
\begin{align*}
\@F|\Sub_\!C(X) : \Sub_\!C(X) -> \Sub_\!D(\@F(X))
\end{align*}
for each $X \in \!C$, which also preserves $\kappa$-ary meets.

A finitely complete category $\!C$ is \defn{extensive} (see \cite{CLW}) if it has finite coproducts which are \defn{pullback-stable} and \defn{disjoint}, defined as follows.  A coproduct $\bigsqcup_{i \in I} X_i$ of objects $X_i \in \!C$ with cocone $(\iota_i : X_i -> \bigsqcup_{i \in I} X_i)_{i \in I}$ is \defn{pullback-stable} if for any morphism $f : Y -> \bigsqcup_{i \in I} X_i \in \!C$, the pullbacks $f^*(\iota_i) : f^*(X_i) -> Y$ of the morphisms $\iota_i$ along $f$, as in the diagram
\begin{equation*}
\begin{tikzcd}
f^*(X_i) \dar["f\vert f^*(X_i)"'] \rar["f^*(\iota_i)"] & Y \dar["f"] \\
X_i \rar["\iota_i"'] & \bigsqcup_i X_i,
\end{tikzcd}
\end{equation*}
together form a cocone which exhibits $Y$ as the coproduct $\bigsqcup_i f^*(X_i)$.  In particular, when $I = \emptyset$, this means we have an initial object $0 \in \!C$ which is \defn{strict}, meaning any $f : Y -> 0$ is an isomorphism.  A coproduct $\bigsqcup_i X_i$ is \defn{disjoint} if each of the cocone morphisms $\iota_i : X_i -> \bigsqcup_{i \in I} X_i$ is monic, and for any $i \ne j$, the pullback of $\iota_i, \iota_j$ is the initial object.  For extensivity, it is enough to consider nullary and binary coproducts, i.e., for there to be a strict initial object, as well as pullback-stable and disjoint binary coproducts.  More generally, we say that $\!C$ is \defn{$\kappa$-extensive} if it has pullback-stable and disjoint $\kappa$-ary coproducts (disjointness needs only be checked for binary coproducts).
A finitely continuous functor $\@F : \!C -> \!D$ between finitely complete ($\kappa$-)extensive categories is \defn{($\kappa$-)extensive} if it preserves finite (resp., $\kappa$-ary) coproducts.

A finitely complete category $\!C$ is \defn{subextensive} if it has a strict initial object $0$, which implies that every subobject poset $\Sub(X)$ has a least element $0 `-> X$, and any two subobjects $A, B \setle X$ which are \defn{disjoint} (meaning their meet is $0$) have a pullback-stable join in $\Sub(X)$.
(This notion is part of the definition of \defn{coherent} category, see e.g., \cite[A1.4]{Jeleph}, as well as an instance of Barr's \defn{effective unions} \cite{Bar}.)
We will denote the pairwise disjoint joins by $A \sqcup B$ (this is consistent with the above usage by \cref{thm:ext-coprod-join,thm:subext-join-coprod} below).  More generally, we say that $\!C$ is \defn{$\kappa$-subextensive} if any $\kappa$-ary family of pairwise disjoint $A_i \setle X$ has a pullback-stable join $\bigsqcup_i A_i$ in $\Sub(X)$.
A finitely continuous functor $\@F : \!C -> \!D$ between finitely complete ($\kappa$-)subextensive categories is \defn{($\kappa$-)subextensive} if it preserves the initial object and finite (resp., $\kappa$-ary) joins of disjoint subobjects.

The subobject posets $\Sub(X)$ in a finitely complete $\kappa$-subextensive category $\!C$ have a partial algebraic structure which we will refer to as a \defn{$\kappa$-disjunctive frame}%
\footnote{The terminology is motivated by \defn{$\kappa$-frames} \cite{Mad}.}
(or \defn{distributive disjunctive lattice} when $\kappa = \omega$), meaning a meet-semilattice with (greatest and) least element and joins of pairwise disjoint $\kappa$-ary families, over which finite meets distribute; distributivity in $\Sub(X)$ follows from pullback-stability.  The pullback maps $f^* : \Sub(Y) -> \Sub(X)$ are $\kappa$-disjunctive frame homomorphisms (giving a functor $\Sub : \!C^\op -> \!{\kappa DisjFrm}$ to the category of $\kappa$-disjunctive frames), as are the maps $\@F|\Sub_\!C(X) : \Sub_\!C(X) -> \Sub_\!D(\@F(X))$ induced by $\kappa$-subextensive functors $\@F : \!C -> \!D$.

The following facts are standard (see \cite[A1.4]{Jeleph}):

%

\begin{lemma}
\label{thm:ext-coprod-join}
A finitely complete $\kappa$-extensive category $\!C$ is $\kappa$-subextensive, with joins of pairwise disjoint subobjects given by their coproduct.  Hence, a finitely continuous functor between finitely complete $\kappa$-extensive categories is $\kappa$-extensive iff it is $\kappa$-subextensive.
\end{lemma}
\begin{proof}
Let $A_i \setle X \in \!C$ for $i \in I$, $\abs{I} < \kappa$, be pairwise disjoint.  It is enough to show that the morphism $h : \bigsqcup_i A_i -> X$ induced by the $A_i `-> X$ via the universal property of the coproduct is monic, for then the universal property easily implies that $\bigsqcup_i A_i$ is the join in $\Sub(X)$, and this join is pullback-stable since the coproduct $\bigsqcup_i A_i$ is.  For this, consider the lower-right pullback square in
\begin{equation*}
\begin{tikzcd}
p^*(A_i) \cap q^*(A_j) \dar[hook] \rar[hook] & q^*(A_j) \dar[hook] \rar & A_j \dar[hook] \\
p^*(A_i) \dar \rar[hook] & (\bigsqcup_i A_i) \times_X (\bigsqcup_i A_i) \dar["p"'] \rar["q"] & \bigsqcup_i A_i \dar["h"] \\
A_i \rar[hook] & \bigsqcup_i A_i \rar["h"'] & X.
\end{tikzcd}
\end{equation*}
By pullback-stability of $\bigsqcup_i A_i$, we may write $(\bigsqcup_i A_i) \times_X (\bigsqcup_i A_i)$ as the coproduct
\begin{align*}
(\bigsqcup_i A_i) \times_X (\bigsqcup_i A_i) = \bigsqcup_i p^*(A_i) = \bigsqcup_{i,j} (p^*(A_i) \cap q^*(A_j)).
\end{align*}
For each $i, j$, $p^*(A_i) \cap q^*(A_j)$ is the pullback of the composites $A_i, A_j \setle \bigsqcup_i A_i --->{h} X$, i.e., the inclusions $A_i, A_j `-> X$, hence is either the diagonal $A_i `-> (\bigsqcup_i A_i) \times_X (\bigsqcup_i A_i)$ if $i = j$, or $0$ if $i \ne j$.  So $(\bigsqcup_i A_i) \times_X (\bigsqcup_i A_i) = \bigsqcup_i A_i$, i.e., $h$ is monic.
\end{proof}

\begin{lemma}
\label{thm:subext-join-coprod}
Let $\!C$ be a finitely complete $\kappa$-subextensive category, $X \in \!C$, and $A_i \setle X$ be a $\kappa$-ary family of pairwise disjoint subobjects.  Then their join $\bigsqcup_i A_i$ in $\Sub(\!C)$ is their coproduct in $\!C$.
\end{lemma}
\begin{proof}
Let $f_i : A_i -> Y \in \!C$; we must find a unique $f : \bigsqcup_i A_i -> Y$ such that $f|A_i = f_i$ for each $i$.  In any finitely complete category, morphisms $f : X -> Y$ are in bijection with their \defn{graphs} $(1_X, f) : X `-> X \times Y$, which are subobjects $G \setle X \times Y$ whose composite with the first projection $X \times Y -> X$ is an isomorphism.  For our desired $f : \bigsqcup_i A_i -> Y$, letting $G \setle (\bigsqcup_i A_i) \times Y$ be its graph, the condition $f|A_i = f_i$ is easily seen to be equivalent to $G \cap (A_i \times Y)$ being the graph $G_i$ of $f_i$ for each $i$.  Since $(\bigsqcup_i A_i) \times Y = \bigsqcup_i (A_i \times Y)$ by pullback-stability of the join $\bigsqcup_i A_i$ (under the projection $(\bigsqcup_i A_i) \times Y -> \bigsqcup_i A_i$), we must have (by distributivity) $G = G \cap \bigsqcup_i (A_i \times Y) = \bigsqcup_i (G \cap (A_i \times Y)) = \bigsqcup_i G_i$, which determines $G$ as a subobject of $(\bigsqcup_i A_i) \times Y$.  It remains only to check that $G$ is a graph.  The composite $G `-> (\bigsqcup_i A_i) \times Y -> \bigsqcup_i A_i$ is monic, since its pullback with itself is the subobject $p^*(G) \cap q^*(G) \setle (\bigsqcup_i A_i) \times Y \times Y$ in the pullback diagram
\begin{equation*}
\begin{tikzcd}
p^*(G) \cap q^*(G) \dar[hook] \rar[hook] & q^*(G) \dar[hook] \rar & G \dar[hook] \\
p^*(G) \dar \rar[hook] & (\bigsqcup_i A_i) \times Y \times Y \dar["p"'] \rar["q"] & (\bigsqcup_i A_i) \times Y \dar \\
G \rar[hook] & (\bigsqcup_i A_i) \times Y \rar & \bigsqcup_i A_i,
\end{tikzcd}
\end{equation*}
and we have $p^*(G) \cap q^*(G) = \bigsqcup_{i,j} (p^*(G_i) \cap q^*(G_j))$ (by pullback-stability of the join $G = \bigsqcup_i G_i$ and distributivity) which is easily seen to be the diagonal of $G = \bigsqcup_i G_i$ using that each $G_i$ is a graph and their domains $A_i \setle \bigsqcup_i A_i$ are disjoint.  And the composite $G `-> (\bigsqcup_i A_i) \times Y -> \bigsqcup_i A_i$ as a subobject is all of $\bigsqcup_i A_i$, since it contains each $A_i$, since $G$ contains $G_i$ and $G_i$ is a graph.
\end{proof}

\begin{corollary}
\label{thm:subext-ext}
A finitely complete $\kappa$-subextensive category $\!C$ in which every $\kappa$-ary family of objects $X_i$ jointly embed as pairwise disjoint subobjects of some other object $Y$ is $\kappa$-extensive, with the coproduct $\bigsqcup_i X_i$ given by the join in $\Sub(Y)$.  \qed
\end{corollary}

In a distributive disjunctive lattice $A$, with least and greatest elements $\bot, \top$ and meets and disjoint joins denoted $\wedge, \vee$, we may already define the familiar (in distributive lattices) notion of \defn{complement} of $a \in A$, i.e., any $b \in A$ such that $a \wedge b = \bot$ and $a \vee b = \top$.  Complements are unique, since if $b, c$ are complements of $a$, then $b = b \wedge (a \vee c) = (b \wedge a) \vee (b \wedge c) = b \wedge c \le c$.  So we may denote the complement of $a$, if it exists, by $\neg a$; and complements are automatically preserved by distributive disjunctive lattice homomorphisms.

\begin{lemma}
A distributive disjunctive lattice $A$ with all complements is a Boolean algebra.
\end{lemma}
\begin{proof}
Let $a, b \in A$.  Then $a \wedge \neg a \wedge b = \bot$, so the join $a \vee (\neg a \wedge b)$ exists, and meets distribute over it.  Clearly, $a \le a \vee (\neg a \wedge b)$, and for any $c \ge a, b$, we have $c \ge a, \neg a \wedge b$ whence $c \ge a \vee (\neg a \wedge b)$; we also have
$b \wedge (a \vee (\neg a \wedge b))
= (b \wedge a) \vee (b \wedge \neg a)
= b \wedge (a \vee \neg a) = b$,
i.e., $b \le a \vee (\neg a \wedge b)$, whence $a \vee (\neg a \wedge b)$ is the join $a \vee b$ of $a, b$.  For any $c \in A$, we have
$c \wedge (a \vee b)
= c \wedge (a \vee (\neg a \wedge b))
= (c \wedge a) \vee (c \wedge \neg a \wedge b)
\le (c \wedge a) \vee (c \wedge b)$,
so $A$ is a distributive lattice.
\end{proof}

We call a finitely complete $\kappa$-subextensive category $\!C$ \defn{Boolean} if each of its subobject $\kappa$-disjunctive frames $\Sub(X)$ is a ($\kappa$-complete) Boolean algebra.  (We will discuss $\kappa$-complete Boolean algebras in more detail in the next section.)  For such $\!C$, the pullback maps between subobject lattices are ($\kappa$-)Boolean homomorphisms; and a finitely continuous $\kappa$-(sub)extensive functor $\@F : \!C -> \!D$ between two such $\!C, \!D$ preserves complements of subobjects.  Note also that if $\!C$ is $\kappa$-complete Boolean subextensive, then it is $\kappa$-subextensive, since its subobject Boolean algebras have $\kappa$-ary meets, hence $\kappa$-ary joins (which are pullback-stable, since $\kappa$-ary meets and complements are).

\begin{lemma}
\label{thm:klimbext-kext}
A $\kappa$-complete Boolean extensive category $\!C$ is $\kappa$-extensive.
\end{lemma}
\begin{proof}
Since $\!C$ is $\kappa$-subextensive, by \cref{thm:subext-ext}, it is enough to show that any $\kappa$-ary family of objects $X_i \in \!C$ can be jointly embedded as pairwise disjoint subobjects of some other object.  Consider $f_i := (f_{ij})_j : X_i -> \prod_j (X_j \sqcup 1)$ for each $i$, where $f_{ii} : X_i `-> X_i \sqcup 1$ is the coproduct injection and $f_{ij} : X_i -> 1 `-> X_j \sqcup 1$ for $j \ne i$.  Each $f_i$ is monic, since its composition with the $i$th projection $\pi_i : \prod_j (X_j \sqcup 1) -> X_i \sqcup 1$ is $f_{ii} : X_i `-> X_i \sqcup 1$; and for $i \ne j$, we have a commutative diagram
\begin{equation*}
\begin{tikzcd}
X_i \dar[equal] \rar["f_i"] & \bigsqcup_k (X_k \sqcup 1) \dar["\pi_i"] & X_j \lar["f_j"'] \dar \\
X_i \rar[hook] & X_i \sqcup 1 & 1 \lar[hook]
\end{tikzcd}
\end{equation*}
whence the pullback of the top row maps into the pullback of the bottom row which is $0$.
\end{proof}

The natural framework for studying categories equipped with algebraic structure (e.g., extensive categories and the above variations) is that of 2-categories.  We will quickly review the basic notions needed in this paper; see \cite[I~Ch.~7]{Bor}, \cite[B1.1]{Jeleph} for more comprehensive background.

A \defn{2-category}, denoted by a Fraktur symbol like $\&C$, consists of objects $\!X \in \&C$, together with for each pair of objects $\!X, \!Y \in \&C$ a \defn{hom-category} $\&C(\!X, \!Y)$, whose objects $\@F \in \&C(\!X, \!Y)$ are called \defn{morphisms} $\@F : \!X -> \!Y$ of $\&C$ and morphisms $\theta : \@F -> \@G \in \&C(\!X, \!Y)$ are called \defn{2-cells} $\theta : \@F -> \@G : \!X -> \!Y$ of $\&C$, together with for each $\!X \in \&C$ an \defn{identity morphism} $1_\!X \in \&C(\!X, \!X)$ and for each $\!X, \!Y, \!Z \in \&C$ a \defn{composition functor} $\&C(\!Y, \!Z) \times \&C(\!X, \!Y) -> \&C(\!X, \!Z)$, obeying the usual asosciaitivity and unitality laws (on the nose, not just up to natural isomorphism of functors).  A \defn{(strict) 2-functor} between 2-categories is a map taking objects to objects, morphisms to morphisms, and 2-cells to 2-cells (with the obvious endpoint compatibility conditions), which is a functor on each hom-category and also preserves the global composition and identity.

The quintessential example of a 2-category is $\&{Cat}$, whose objects are (small)%
\footnote{We will generally ignore size issues, which can be dealt with via standard tricks, since all the algebraic structures we consider will be bounded in arity by $\kappa$.}
categories, morphisms are functors, and 2-cells are natural transformations; that is, each hom-category $\&{Cat}(\!C, \!D)$ is the functor category $\!D^\!C$.  We will consider the following sub-2-categories of $\&{Cat}$, consisting of categories with algebraic structure and functors preserving said structure.
Let $\kappa\&{LimCat} \subseteq \&{Cat}$ be the sub-2-category of $\kappa$-complete categories, $\kappa$-continuous functors, and (all) natural transformations.
Let
\begin{equation*}
\kappa\&{Lim(B)(\kappa)(S)ExtCat} \subseteq \kappa\&{LimCat} \subseteq \&{Cat}
\end{equation*}
be the further sub-2-categories of $\kappa$-complete (Boolean) ($\kappa$-)(sub)extensive categories, $\kappa$-continuous ($\kappa$-)(sub)extensive functors, and natural transformations.  For example, $\kappa\&{LimBSExtCat}$ is the 2-category of $\kappa$-complete Boolean subextensive categories and $\kappa$-continuous subextensive functors.  These are related as follows:
\begin{equation*}
\begin{tikzcd}
\kappa\&{LimB\kappa ExtCat} \ar[dd, tail] \ar[dr, equal] \ar[rr, tail] &[-4em]&[-4em] \kappa\&{LimB\kappa SExtCat} \ar[dd, tail] \ar[dr, equal] &[-4em] \\
& \kappa\&{LimBExtCat} \ar[rr, tail, crossing over] && \kappa\&{LimBSExtCat} \ar[dd, tail] \\
\kappa\&{Lim\kappa ExtCat} \ar[dr, hook] \ar[rr, tail] && \kappa\&{Lim\kappa SExtCat} \ar[dr, hook] \\
& \kappa\&{LimExtCat} \ar[uu, leftarrowtail, crossing over] \ar[rr, tail] && \kappa\&{LimSExtCat} \rar[hook] & \kappa\&{LimCat} \rar[hook] & \&{Cat}
\end{tikzcd}
\end{equation*}
The $\rightarrowtail$'s denote \defn{full} sub-2-categories, i.e., the sub-2-category includes all functors in the bigger 2-category between two objects in the sub-2-category.  For the horizontal $\rightarrowtail$'s, fullness is by \cref{thm:ext-coprod-join}.  For the vertical $\rightarrowtail$'s, fullness is by definition (subextensive functors automatically preserve complements of subobjects).  The diagonal equalities in the top layer are by \cref{thm:klimbext-kext} and the sentence preceding it; we will use the shorter names for these two 2-categories, which are the focus of this paper.

Because categorical structure (e.g., limits) is often only defined up to canonical isomorphism, the ``correct'' 2-categorical analog of a (1-)categorical notion often involves weakening a naive analog by replacing some equalities with coherent isomorphisms.  For example, the notion of (strict) 2-functor from above, while useful, is not general enough to include many naturally occurring examples, and must be weakened to the notion of \defn{pseudofunctor} (where composition is preserved only up to coherent isomorphism; see \cite[I~7.5]{Bor}, \cite[B1.1]{Jeleph}).  Similarly, the universal property of limits and colimits must be weakened by relaxing the uniqueness condition in order to arrive at what are commonly called \defn{bi(co)limits} (see \cite[I~7.4]{Bor}).  Indeed, it is known that 2-categories of categories with ``nice'' algebraic structure (such as those above) generally lack colimits in the naive (strict) sense, although they admit all bicolimits (see \cite{BKP}).

In this paper, we will only need one simple instance of a non-strict 2-categorical notion: that of a \defn{(bi-)initial object} (we henceforth drop the ``bi'' prefix) $\!X$ in a 2-category $\&C$, which is an object such that for any other $\!Y \in \&C$, the hom-category $\&C(\!X, \!Y)$ is equivalent to the terminal category.  This means that for any other $\!Y$, there is a morphism $\!X -> \!Y$, and for any two such morphisms $\@F, \@G : \!X -> \!Y$, there is a unique 2-cell $\@F -> \@G$ (which must thus be an isomorphism).  Equivalently, this means $\!X$ admits a morphism to any other object, and this morphism is unique up to unique (2-cell) isomorphism and has no non-identity 2-cells to itself.  In particular, an initial object $\!X \in \&C$ is unique up to equivalence in $\&C$ (with the equivalence being itself unique up to unique isomorphism).

For example, an initial object in $\kappa\&{LimBExtCat}$, as provided by \cref{thm:intro-kboolk-init}, is a $\kappa$-complete Boolean extensive category $\!C$, such that for any other $\kappa$-complete Boolean extensive category $\!D$, there is a $\kappa$-continuous extensive functor $\@F : \!C -> \!D$, which is unique up to unique natural isomorphism, and admits no non-identity natural transformations $\@F -> \@F$.  Such $\!C$, if it exists, is unique up to (unique-up-to-unique-isomorphism) equivalence of categories.

\section{$\kappa$-Boolean algebras}
\label{sec:kbool}

For general background on Boolean algebras, see e.g., \cite{Sik}.  We denote Boolean operations by $\wedge, \vee, \neg, \top, \bot$, and also use the implication operation $a -> b := \neg a \vee b$ and bi-implication $a <-> b := (a -> b) \wedge (b -> a)$.  We will never denote $\top, \bot$ by $1, 0$ (which are reserved for generating elements of free algebras over finite cardinals).

A \defn{$\kappa$-(complete )Boolean algebra} is a Boolean algebra with $\kappa$-ary joins (hence $\kappa$-ary meets).  Let $\kappa\!{Bool}$ denote the category of $\kappa$-Boolean algebras and ($\kappa$-join-preserving) homomorphisms.  $\kappa$-Boolean algebras are defined by a $\kappa$-ary (infinitary) algebraic theory, whence $\kappa\!{Bool}$ is a well-behaved ``algebraic'' (i.e., monadic over $\!{Set}$) category, in particular having all small limits and colimits (see e.g., \cite[II~4.3]{Bor}).
We will denote binary coproducts in $\kappa\!{Bool}$ by $A \otimes B$, and more generally, the pushout of $B, C$ over $A$ by $B \otimes_A C$.%
\footnote{This is only by analogy with rings, unless $\kappa = \omega$ in which case we can identify Boolean algebras with Boolean rings and take the usual tensor product (= coproduct) of commutative rings.}

For each set $X$, let $\@K(X)$ denote the free $\kappa$-Boolean algebra generated by $X$.  We identify the elements of $X$ with the generators in $\@K(X)$, so that $X \setle \@K(X)$.  For a map $f : X -> Y$, let $\@K(f) = f_* : \@K(X) -> \@K(Y)$ be the $\kappa$-Boolean extension of $f$.  Thus $\@K : \!{Set} -> \kappa\!{Bool}$ is the free $\kappa$-Boolean algebra functor, left adjoint to the forgetful functor $\kappa\!{Bool} -> \!{Set}$.

Let $\kappa\!{Bool}_\kappa \setle \kappa\!{Bool}$ denote the full subcategory of \defn{$\kappa$-presented} $\kappa$-Boolean algebras, i.e., $\kappa$-Boolean algebras (isomorphic to ones) of the form $\@K(X)/{\sim}$ for some $X \in \!{Set}_\kappa$ and some $\kappa$-generated $\kappa$-Boolean algebra congruence relation ${\sim} \setle \@K(X)^2$.

The following algebraic fact is standard:

\begin{lemma}
\label{thm:alexandrov}
Let $A$ be a $\kappa$-presented $\kappa$-Boolean algebra, and $X \setle A$ be $<\kappa$-many generators.  Then $A$ has a $\kappa$-ary presentation using only the generators in $X$, i.e., the congruence kernel of the canonical homomorphism $p : \@K(X) ->> A$ is $\kappa$-generated.
\end{lemma}
\begin{proof}
Let $q : \@K(Y) ->> \@K(Y)/{\sim} \cong A$ be some $\kappa$-ary presentation, with ${\sim} = \ker(q)$ $\kappa$-generated, say by pairs $s_i \sim t_i$ for $i \in I$ ($\abs{I} < \kappa$) where $s_i, t_i \in \@K(Y)$.  For each $x \in X$, pick some $u(x) \in q^{-1}(x) \setle \@K(Y)$, and for each $y \in Y$, pick some $v(y) \in p^{-1}(y) \setle \@K(X)$.  Extend $u, v$ to homomorphisms $u : \@K(X) -> \@K(Y)$ and $v : \@K(Y) -> \@K(X)$.  Then it is straightforward to see that $\ker(p) = {\approx}$, where $\approx$ is the congruence on $\@K(X)$ generated by $v(s_i) \approx v(t_i)$ and $x \approx v(u(x))$.
\end{proof}

It is a standard fact that a congruence ${\sim}$ on a $\kappa$-Boolean algebra $A$ is determined by the congruence class $[\top]_\sim \setle A$, which is an arbitrary $\kappa$-filter on $A$; and $\sim$ is $\kappa$-generated iff $[\top]_\sim$ is a principal filter $\up a$ (in which case $\sim$ is generated by the single pair $\top \sim a$, and we can take $a := \bigwedge_i (b_i <-> c_i)$ for some $\kappa$-ary family of generators $b_i \sim c_i$ of $\sim$).  For $a \in A$, we denote by $A/a$ the quotient by the congruence $\ang{(\top, a)} \setle A^2$ corresponding to the principal filter $\up a$.  Note that $A/a$ can be identified with the principal ideal $\down a \setle A$ via the isomorphism $\down a `-> A ->> A/a$; the quotient map is then identified with $a \wedge (-) : A ->> \down a$.  Thus for $\kappa$-presented $A$, we have bijections
\begin{alignat*}{2}
A &\cong \{\text{$\kappa$-generated congruences on } A\} &&\cong \{\text{$\kappa$-presented quotients of } A\} \\
a &|-> \ang{(\top, a)} &&|-> A/a \cong \down a
\end{alignat*}
where the second $\cong$ is by \cref{thm:alexandrov}, and the first $\cong$ is order-reversing.  Note, furthermore, that these bijections are compatible with homomorphisms, in the sense that for $f : A -> B \in \kappa\!{Bool}_\kappa$ and $a \in A$, $B/f(a)$ is the pushout of $f$ and $A/a$:
\begin{equation*}
\begin{tikzcd}
A \dar[two heads] \rar["f"] & B \dar[two heads] \\
A/a \rar & B/f(a) = B \otimes_A A/a.
\end{tikzcd}
\end{equation*}

Recall that in algebraic categories such as $\kappa\!{Bool}$, surjective homomorphisms are precisely the \defn{regular epimorphisms} (epimorphisms which are the coequalizer of some parallel pair of morphisms; see e.g., \cite[II~4.3.5]{Bor}).  Lagrange \cite{Lag} showed that in $\kappa\!{Bool}$, all epimorphisms are regular.  A consequence of his proof is the following key result, in some sense the main technical ingredient of this paper:

\begin{theorem}[Lagrange interpolation theorem]
\label{thm:lagrange}
Let
\begin{equation*}
\begin{tikzcd}
A \dar["g"'] \rar["f"] & B \dar["g'"] \\
C \rar["f'"'] & B \otimes_A C
\end{tikzcd}
\end{equation*}
be a pushout in $\kappa\!{Bool}$.  Then for any $b \in B$ and $c \in C$ such that $g'(b) \le f'(c)$, there is $a \in A$ such that $b \le f(a)$ and $g(a) \le c$.
\end{theorem}
\begin{proof}
Lagrange \cite{Lag} proves that $\kappa\!{Bool}$ has the \defn{(strong) amalgamation property}, i.e., that if $f, g$ above are injective, then so are $f', g'$ (and the diagram is a pullback).  While his argument can easily be adapted to yield the above statement, we can also deduce it abstractly, as follows.
Let
\begin{align*}
U &:= \{a_1 \wedge \neg a_2 \mid b \le f(a_1) \AND g(a_2) \le c\} \subseteq A, \\
V &:= \{b_1 \wedge \neg f(a_2) \mid b \le b_1 \AND g(a_2) \le c\} \subseteq B, \\
W &:= \{g(a_1) \wedge \neg c_2 \mid b \le f(a_1) \AND c_2 \le c\} \subseteq C;
\end{align*}
these are $\kappa$-filters.  It is easily seen that $U = f^{-1}(V) = g^{-1}(W)$,
whence replacing $A, B, C$ with $A/U, B/V, C/W$ respectively in the diagram above renders $f, g$ injective.  Clearly $b \in V$ and $\neg c \in W$, i.e., $b, c$ become $\top, \bot$ in $B/V, C/W$ respectively; thus $g'(b) \le f'(c) \in B \otimes_A C$ implies that $B/V \otimes_{A/U} C/W$ is trivial.  Now by amalgamation, $A/U$ must be trivial, i.e., $\bot \in U$, i.e., there are $a_1 \le a_2 \in A$ such that $b \le f(a_1)$ and $g(a_2) \le c$.  Taking $a := a_1$ works.
\end{proof}

\begin{corollary}[Lagrange]
\label{thm:kbool-epi}
All epimorphisms in $\kappa\!{Bool}$ are surjective.
\end{corollary}
\begin{proof}
Let $f : A -> B \in \kappa\!{Bool}$ be an epimorphism, i.e., the above diagram with $g = f$, $C = B = B \otimes_A C$, and $f' = g' = 1_B$ is a pushout.  Then for any $b \in B$, by \cref{thm:lagrange}, there is $a \in A$ such that $b \le f(a) \le b$.
\end{proof}

\begin{corollary}
\label{thm:kboolk-quotient}
For any $\kappa$-presented $\kappa$-Boolean algebra $A$, we have an order-isomorphism
\begin{align*}
A &\cong \Sub_{\kappa\!{Bool}_\kappa^\op}(A) \\
a &|-> A/a \cong \down a 
\end{align*}
which as $A$ varies is a natural transformation $1_{\kappa\!{Bool}_\kappa} -> \Sub_{\kappa\!{Bool}_\kappa^\op} : \kappa\!{Bool}_\kappa -> \kappa\!{Bool}_\kappa$, i.e., a homomorphism $f : A -> B \in \kappa\!{Bool}_\kappa$ corresponds to pullback of subobjects $\Sub_{\kappa\!{Bool}_\kappa^\op}(A) -> \Sub_{\kappa\!{Bool}_\kappa^\op}(B)$ along $f : B -> A \in \kappa\!{Bool}_\kappa^\op$.
\end{corollary}
\begin{proof}
By \cref{thm:kbool-epi} and the discussion preceding \cref{thm:lagrange}.
\end{proof}

We will also need the following more technical consequence of \cref{thm:lagrange}:

\begin{lemma}
\label{thm:lusin-suslin}
Let $Z \subseteq X$ be $\kappa$-ary sets, let
\begin{equation*}
\begin{tikzcd}
Z \dar[hook, "i"'] \rar[hook, "i"] & X \dar[hook, "g"] \\
X \rar[hook, "f"'] & X \sqcup_Z X
\end{tikzcd}
\end{equation*}
be a pushout in $\!{Set}_\kappa$, where $i$ is the inclusion, and consider the induced pushout
\begin{equation*}
\begin{tikzcd}
\@K(Z) \dar[hook, "i_*"'] \rar[hook, "i_*"] & \@K(X) \dar[hook, "g_*"] \\
\@K(X) \rar[hook, "f_*"'] & \@K(X \sqcup_Z X) \mathrlap{{}\cong \@K(X) \otimes_{\@K(Z)} \@K(X)}
\end{tikzcd}
\end{equation*}
in $\kappa\!{Bool}_\kappa$.  For any $b \in \@K(X)$ such that
\begin{align*}
\tag{$*$}
f_*(b) \wedge g_*(b) \le \bigwedge_{y \in X \setminus Z} (f(y) <-> g(y)) \in \@K(X \sqcup_Z X),
\end{align*}
there is a retraction $h : \@K(X) ->> \@K(Z) \in \kappa\!{Bool}_\kappa$ of $i_* : \@K(Z) `-> \@K(X)$ such that
\begin{align*}
\tag{$\dagger$}
b = i_*(h(b)) \wedge \bigwedge_{y \in X \setminus Z} (i_*(h(y)) <-> y) \in \@K(X).
\end{align*}
\end{lemma}
\begin{proof}
On generators $z \in Z \subseteq X \subseteq \@K(X)$, we must put $h(z) := z \in \@K(Z)$ to guarantee that $h$ is a retraction of $i_*$.  For $y \in X \setminus Z$, by ($*$), we have
\begin{align*}
f_*(b \wedge y) \le g_*(b -> y),
\end{align*}
whence by \cref{thm:lagrange} there is some $h(y) \in \@K(Z)$ such that
\begin{align*}
b \wedge y \le i_*(h(y)) \le b -> y,
\qquad\text{or equivalently,}\qquad
b \wedge y = b \wedge i_*(h(y)).
\tag{$\ddagger$}
\end{align*}
This defines $h$ on generators $X \subseteq \@K(X)$.

Note that ($\ddagger$), which also trivially holds for $y \in Z$, says precisely that the homomorphism $i_* \circ h : \@K(X) -> \@K(X)$ becomes equal (on generators $y \in X$) to the identity after projecting to the quotient $\@K(X)/b \cong \down b$.  Thus, ($\ddagger$) holds more generally for \emph{all} $y \in \@K(X)$.

We now verify ($\dagger$).  For $\le$: $b \le i_*(h(b))$ follows from ($\ddagger$) by taking $y := b$, while $b \le i_*(h(y)) <-> y$ is a restatement of ($\ddagger$).  For $\ge$: the right-hand side is the projection of $i_*(h(b))$ to the quotient $\@K(X)/\bigwedge_{y \in X/Z} (i_*(h(y)) <-> y) \cong \down \bigwedge_{y \in X/Z} (i_*(h(y)) <-> y)$ identifying each $y \in X \setminus Z$ with $i_*(h(y))$; clearly $y \in Z$ is also identified with $i_*(h(y)) = y$, whence again $i_* \circ h$ becomes equal to the identity after projecting to this quotient, whence
\begin{align*}
i_*(h(b)) \wedge \bigwedge_{y \in X/Z} (i_*(h(y)) <-> y) = b \wedge \bigwedge_{y \in X/Z} (i_*(h(y)) <-> y) \le b.
&\qedhere
\end{align*}
\end{proof}

\begin{remark}
\label{rmk:lusin-suslin}
When $\kappa = \omega_1$, the above results (when all the algebras are countably presented) correspond, via Loomis--Sikorski duality as in the following section, to familiar results in descriptive set theory.  In \cref{thm:lagrange}, taking $A, B, C$ to be the duals of standard Borel spaces $X, Y, Z$ respectively and $f, g$ to be the duals of Borel maps $p : Y -> X$ and $q : Z -> X$, the result says that for any Borel sets $b \subseteq Y$ and $c \subseteq Z$ such that $b \times_X Z \subseteq Y \times_X c$, i.e., $p(b) \cap q(\neg c) = \emptyset$, there is a Borel set $a \subseteq X$ such that $b \subseteq p^{-1}(a)$ and $q^{-1}(a) \subseteq c$, i.e., $p(b) \subseteq a$ and $q(\neg c) \cap a = \emptyset$.  This is the Lusin separation theorem for $\*\Sigma^1_1$ sets (see \cite[14.7]{Kcdst}).

Both \cref{thm:kboolk-quotient} and \cref{thm:lusin-suslin} are dual versions of the Lusin--Suslin theorem on injective Borel images (see \cite[15.1]{Kcdst}).  \Cref{thm:kboolk-quotient} says that for any injective Borel map $p : Y -> X$ (dual to an epimorphism $A -> B \in \sigma\!{Bool}_\sigma$), there is a Borel set $a \subseteq X$ such that $p$ is an isomorphism between $Y$ and the Borel subspace $a$ (dual to the quotient $A/a$).  \Cref{thm:lusin-suslin} says that for any countable sets $X = Y \sqcup Z$ and Borel set $b \subseteq 2^X \cong 2^Z \times 2^Y$ such that $b \times_{2^Z} b \subseteq 2^X \times_{2^Z} 2^X \cong 2^Z \times 2^Y \times 2^Y$ is contained in the diagonal, i.e., the projection $2^Z \times 2^Y ->> 2^Z$ restricted to $b$ is injective, there is a Borel map $u : 2^Z -> 2^Y$ whose graph contains $b$.  (In fact, both proofs, once unravelled, can be seen as dual versions of the proof of Lusin--Suslin given in \cite{Cls}.  See also \cite[3.4.26]{Cloc}.)
\end{remark}

In the rest of this section, we verify that $\kappa\!{Bool}_\kappa^\op$ is $\kappa$-complete Boolean $\kappa$-extensive.  Clearly $\kappa\!{Bool}_\kappa \setle \kappa\!{Bool}$ is closed under $\kappa$-ary colimits, whence $\kappa\!{Bool}_\kappa^\op$ is $\kappa$-complete.  Booleanness follows from \cref{thm:kboolk-quotient}.  Thus, it remains to show $\kappa$-extensivity.%
\footnote{These proofs become much more motivated when one thinks of $\kappa\!{Bool}_\kappa^\op$ as standard $\kappa$-Borel locales, as sketched in the introduction.}

\begin{proposition}
\label{thm:kboolk-prod}
$\kappa$-ary products of $\kappa$-presented $\kappa$-Boolean algebras are $\kappa$-presented.  Thus, $\kappa\!{Bool}_\kappa^\op$ has $\kappa$-ary coproducts.
\end{proposition}
\begin{proof}
It is standard that when $<\kappa$-many elements $c_i \in A$ of a $\kappa$-ary Boolean algebra $A$ form a \defn{partition}, i.e., $\bigvee_i c_i = \top$ and $c_i \wedge c_j = \bot$ for $i \ne j$, then we have an isomorphism
\begin{align*}
A &\cong \prod_i \down c_i \cong \prod_i A/c_i \\
a &|-> (c_i \wedge a)_i.
\end{align*}
Now let $A_i = \@K(X_i)/r_i$ with $\abs{X_i} < \kappa$ and $r_i \in \@K(X_i)$ for each $i \in I$, $\abs{I} < \kappa$, be $\kappa$-presented $\kappa$-Boolean algebras.  Then we claim
\begin{align*}
\prod_i A_i \cong \@K(I \sqcup \bigsqcup_i X_i)/s &&\text{where}&& s := \bigwedge_{i \ne j} \neg(i \wedge j) \wedge (\bigvee_i i) \wedge \bigwedge_i \bigwedge_{x \in X_i} (x -> i) \wedge \bigwedge_i (i -> r_i).
\end{align*}
Indeed, the relations in $s$ ensure that the generators $i \in I \setle \@K(I \sqcup \bigsqcup_j X_j)$ form a partition in $\@K(I \sqcup \bigsqcup_j X_j)/s$, and that for each $i \in I$ we have
\begin{align*}
\@K(I \sqcup \bigsqcup_j X_j)/s/i
&\cong \@K(I \sqcup \bigsqcup_j X_j)/(s \wedge i) \\
&= \@K(I \sqcup \bigsqcup_j X_j)/(i \wedge \bigwedge_{j \ne i} (\neg j \wedge \bigwedge_{x \in X_j} \neg x) \wedge r_i) \\
&\cong \@K(X_i)/r_i
= A_i.
\qedhere
\end{align*}
\end{proof}

\begin{proposition}
\label{thm:kboolk-ext}
$\kappa\!{Bool}^\op$ and $\kappa\!{Bool}_\kappa^\op$ are $\kappa$-extensive.
\end{proposition}
\begin{proof}
Clearly, product projections in $\kappa\!{Bool}$ are surjective; and for a product $\prod_k A_k$ and $i \ne j$, the pushout of the projections $\pi_i : \prod_k A_k -> A_i$ and $\pi_j : \prod_k A_k -> A_j$ is trivial, as witnessed by any $\vec{a} \in \prod_k A_k$ with $a_i = \bot$ and $a_j = \top$.  Thus, $\kappa$-ary coproducts in $\kappa\!{Bool}^\op$ are disjoint.

It remains to verify pullback-stability.  Let $A_i \in \kappa\!{Bool}$ for $i \in I$, $\abs{I} < \kappa$, and $f : \prod_i A_i -> B \in \kappa\!{Bool}$, let $\pi_i : \prod_k A_k -> A_i$ be the projections, and let
\begin{equation*}
\begin{tikzcd}
\prod_k A_k \dar["f"'] \rar["\pi_i"] & A_i \dar["f_i"] \\
B \rar["g_i"'] & A_i \otimes_{\prod_k A_k} B
\end{tikzcd}
\end{equation*}
be pushouts; we must show that the $g_i$ exhibit $B$ as the product
\begin{align*}
B = \prod_i (A_i \otimes_{\prod_k A_k} B).
\end{align*}
For each $a \in A_i$, let $\delta_i(a) \in \prod_k A_k$ be given by $\delta_i(a)_i := a$ and $\delta_i(a)_j := \bot$ for all $j \ne i$.  Then $(\delta_i(\top))_i$ form a partition in $\prod_k A_k$, whence $(f(\delta_i(\top)))_i$ form a partition in $B$.
Each $g_i(f(\delta_i(\top))) = f_i(\pi_i(\delta_i(\top))) = f_i(\top) = \top$, whence $g_i$ factors through $B/f(\delta_i(\top)) \cong \down f(\delta_i(\top))$ as the restriction $g_i|\down f(\delta_i(\top)) : \down f(\delta_i(\top)) -> A_i \otimes_{\prod_k A_k} B$.  So it is enough to show that each $g_i|\down f(\delta_i(\top))$ is an isomorphism, since then $g$ factors as
$B \cong \prod_i \down f(\delta_i(\top)) \cong \prod_i (A_i \otimes_{\prod_k A_k} B)$.  We claim that $g_i|\down f(\delta_i(\top))$ has inverse defined via the universal property of the pushout by
\begin{align*}
h_i : A_i \otimes_{\prod_k A_k} B &--> \down f(\delta_i(\top)) \\
f_i(a) &|--> f(\delta_i(a)) \\
g_i(b) &|--> f(\delta_i(\top)) \wedge b.
\end{align*}
$h_i$ is well-defined, since for $\vec{a} \in \prod_k A_k$ we have
\begin{align*}
h_i(f_i(\pi_i(\vec{a})))
&:= f(\delta_i(\pi_i(\vec{a}))) \\
&= f(\delta_i(a_i)) \\
&= f(\delta_i(\top) \wedge \vec{a}) \\
&= f(\delta_i(\top)) \wedge f(\vec{a})
=: h_i(g_i(f(\vec{a}))).
\end{align*}
For all $b \in \down f(\delta_i(\top))$, we have $h_i(g_i(b)) = f(\delta_i(\top)) \wedge b = b$; thus $h_i \circ g_i|\down f(\delta_i(\top)) = 1_{\down f(\delta_i(\top))}$.
For all generators $f_i(a) \in A_i \otimes_{\prod_k A_k} B$ and $g_i(b) \in A_i \otimes_{\prod_k A_k} B$, we have
\begin{align*}
g_i(h_i(f_i(a)))
&= g_i(f(\delta_i(a))) \\
&= f_i(\pi_i(\delta_i(a))) \\
&= f_i(a), \\
g_i(h_i(g_i(b)))
&= g_i(f(\delta_i(\top)) \wedge b) \\
&= g_i(f(\delta_i(\top))) \wedge g_i(b) \\
&= f_i(\pi_i(\delta_i(\top))) \wedge g_i(b) \\
&= f_i(\top) \wedge g_i(b) \\
&= g_i(b);
\end{align*}
thus $g_i \circ h_i = 1_{A_i \otimes_{\prod_k A_k} B}$.
\end{proof}

\begin{corollary}
\label{thm:kboolk-klimbext}
$\kappa\!{Bool}_\kappa^\op$ is a $\kappa$-complete Boolean $\kappa$-extensive category.  \qed
\end{corollary}

\section{Loomis--Sikorski duality}
\label{sec:loomis-sikorski}

In this section, we review the Stone-type duality between standard Borel spaces and countably presented Boolean $\sigma$-algebras.  This duality is essentially a reformulation of the Loomis--Sikorski representation theorem for Boolean $\sigma$-algebras, of which we include a self-contained proof.  This section is the only part of the paper that uses classical descriptive set theory (for which see \cite{Kcdst}).

An $\omega_1$-Boolean algebra is conventionally called a \defn{Boolean $\sigma$-algebra}.  In this context, we generally replace $\kappa$ with $\sigma$ in the notation and terminology of the preceding section; for example, we have the category $\sigma\!{Bool}_\sigma := \omega_1\!{Bool}_{\omega_1}$ of countably presented Boolean $\sigma$-algebras, the notion of $\sigma$-filter, etc.

A \defn{Borel space} (or \defn{measurable space}) is a set $X$ equipped with a (concrete) $\sigma$-algebra $\@B(X)$, i.e., a Boolean $\sigma$-subalgebra $\@B(X) \setle \@P(X)$ of the powerset, whose elements are called \defn{Borel sets} in $X$.  A \defn{Borel map} (or \defn{measurable map}) $f : X -> Y$ between two Borel spaces is a map such that for every Borel set $B \in \@B(Y)$, the preimage $f^{-1}(B)$ is in $\@B(X)$.  Letting $\!{Bor}$ denote the category of Borel spaces and Borel maps, we thus have a functor $\@B : \!{Bor}^\op -> \sigma\!{Bool}$, where $\@B(f) := f^{-1} : \@B(Y) -> \@B(X)$ for a Borel map $f : X -> Y \in \!{Bor}$.

For an arbitrary Boolean $\sigma$-algebra $A$, let $\@S_\sigma(A)$ denote the set of \defn{$\sigma$-ultrafilters} on $A$, i.e., $\sigma$-filters $U \setle A$ whose complements are $\sigma$-ideals.  These are equivalently the preimages of $\top \in 2$ under ($\sigma$-)homomorphisms $A -> 2$.  We equip $\@S_\sigma(A)$ with the Borel $\sigma$-algebra consisting of the sets
\begin{equation*}
[a] := \{U \in \@S_\sigma(A) \mid a \in U\}
\end{equation*}
for all $a \in A$.  In other words, the map $[-] : A -> \@P(\@S_\sigma(A))$ is easily seen to be a $\sigma$-homomorphism; $\@B(\@S_\sigma(A))$ is by definition its image.  For a homomorphism $f : A -> B \in \sigma\!{Bool}$, taking preimage under $f$ takes $\sigma$-ultrafilters on $B$ to $\sigma$-ultrafilters on $A$, yielding a map $\@S_\sigma(f) := f^{-1} : \@S_\sigma(B) -> \@S_\sigma(A)$; we have
$\@S_\sigma(f)^{-1}([a])
= [f(a)]$,
whence $\@S_\sigma(f)$ is a Borel map.  We thus have a functor $\@S_\sigma : \sigma\!{Bool}^\op -> \!{Bor}$.

The preceding two paragraphs are part of the general setup of a Stone-type duality between the categories $\sigma\!{Bool}$ and $\!{Bor}$, induced by the Borel structure and Boolean $\sigma$-algebra structure on the set $2$ which ``commute'' with each other; see \cite[VI~\S4]{Jstone} for the general theory of such dualities.  For an arbitrary Boolean $\sigma$-algebra $A$ and Borel space $X$, a Boolean $\sigma$-homomorphism $f : A -> \@B(X)$ is determined by the set
\begin{equation*}
\Gamma_f := \{(a, x) \in A \times X \mid x \in f(a)\},
\end{equation*}
which is required to have vertical fibers which are Borel subsets of $X$ and horizontal fibers which are $\sigma$-ultrafilters on $A$, which are the same conditions on the set
\begin{equation*}
\Gamma^g := \{(a, x) \in A \times X \mid a \in g(x)\}
\end{equation*}
determining a Borel map $g : X -> \@S_\sigma(A)$.  Thus we have a bijection
\begin{align*}
\sigma\!{Bool}(A, \@B(X)) &\cong \!{Bor}(X, \@S_\sigma(A)) \\
f &|-> (x |-> \{a \in A \mid x \in f(a)\}) \\
(a |-> \{x \in X \mid a \in g(x)\}) &<-| g
\end{align*}
which is easily seen to be natural in $A, X$, yielding a contravariant adjunction between the functors $\@B : \!{Bor}^\op -> \sigma\!{Bool}$ and $\@S_\sigma : \sigma\!{Bool}^\op -> \!{Bor}$.

The adjunction unit on the $\!{Bor}$ side consists of the Borel maps
\begin{align*}
\eta_X : X &--> \@S_\sigma(\@B(X)) = \{\text{$\sigma$-ultrafilters of Borel sets on } X\} \\
x &|--> \{B \in \@B(X) \mid x \in B\}
\end{align*}
for each Borel space $X$, taking points to principal ultrafilters; $\eta_X$ is neither injective nor surjective in general.%
\footnote{For example, consider the Borel equivalence relation $\#E_0$ of equality mod finite on $X := 2^\#N$, and equip $X$ with the $\sigma$-algebra of $\#E_0$-invariant Borel sets.  Then $\eta(x) = \eta(y) \iff x \mathrel{\#E_0} y$, while the complement of the image of $\eta$ contains the conull ultrafilter (by ergodicity) for the standard product measure on $2^\#N$.}
On the $\sigma\!{Bool}$ side, the unit consists of the homomorphisms from above
\begin{equation*}
[-] = [-]_A : A -> \@B(\@S_\sigma(A))
\end{equation*}
for each $A \in \sigma\!{Bool}$.  $[-]_A$ is surjective by definition, and injective iff $A$ admits enough $\sigma$-ultrafilters to separate points, or equivalently (by considering implications $a -> b$) every $a < \top \in A$ is outside of some $\sigma$-ultrafilter; such $A$ are precisely (isomorphic copies of) the concrete $\sigma$-algebras of sets, i.e., $\@B(X)$ for $X \in \!{Bor}$.  In other words, the adjunction between $\@B, \@S_\sigma$ is \defn{idempotent} (see e.g., \cite[VI~4.5]{Jstone}), hence restricts to an adjoint equivalence
\begin{equation*}
\sigma\!{Alg}^\op \cong \!{Bor}'
\end{equation*}
between the full subcategories $\sigma\!{Alg} \setle \sigma\!{Bool}$ of (isomorphic copies of) concrete $\sigma$-algebras, and $\!{Bor}' \setle \!{Bor}$ of Borel spaces $X$ such that $\eta_X$ is an isomorphism.

The classical \defn{Loomis--Sikorski representation theorem} (see e.g., \cite[29.1]{Sik}) states that every Boolean $\sigma$-algebra is a \emph{quotient} of a concrete $\sigma$-algebra.  By standard universal algebra,%
\footnote{$\Longrightarrow$: In particular, every free algebra $\@K(X)$ is a quotient of a concrete $\sigma$-algebra; by projectivity of free algebras, $\@K(X)$ is then itself a concrete $\sigma$-algebra, whence so is every countably presented quotient $\@K(X)/a \cong \down a$.

$\Longleftarrow$: For any free algebra $\@K(X)$, for any $a < \top \in \@K(X)$, letting $Y \setle X$ be all countably many generators appearing in $a$, we get a $\sigma$-ultrafilter $U \setle \@K(Y)$ not containing $a$, which easily extends to a $\sigma$-ultrafilter on $\@K(X)$; this shows that $\@K(X)$ is concrete, which is enough since every algebra is a quotient of a free one.}
this is equivalent to the following formulation, which is more relevant for our purposes.  We include a self-contained proof for the reader's convenience.%
\footnote{This proof is by reduction to the Rasiowa--Sikorski lemma.  Another (easy) reduction is to the completeness theorem for $\@L_{\omega_1\omega}$.  All three results are closely related, and ultimately boil down to a Baire category argument.}

\begin{theorem}[Loomis--Sikorski]
$\sigma\!{Bool}_\sigma \setle \sigma\!{Alg}$, i.e., every countably presented Boolean $\sigma$-algebra admits enough $\sigma$-ultrafilters to separate points.
\end{theorem}
\begin{proof}
Let $\@K(X)/r \cong \down r \in \sigma\!{Bool}_\sigma$, with $X$ countable and $r \in \@K(X)$.  Let $a \in \down r$ with $a < \top_{\down r} = r$; we must find a $\sigma$-ultrafilter $U \setle \down r$ not containing $a$.  Since $a < r$, $r -> a < \top$; so it is enough to find $U \in \@S_\sigma(\@K(X))$ not containing $r -> a$, since then $U \cap \down r \in \@S_\sigma(\down r)$ not containing $a$.  So we may assume to begin with that $r = \top$ and $a < \top_{\@K(X)}$, and find $U \in \@S_\sigma(\@K(X))$ not containing $a$.

Let $s$ be a Boolean $\sigma$-algebra term (i.e., expression built from $\bigwedge, \bigvee, \neg, \top, \bot$) in $X$ evaluating to $a \in \@K(X)$; for convenience, we regard $\bigwedge$ as an abbreviation for $\neg \bigvee \neg$.  Let $T$ be a countable set of Boolean $\sigma$-algebra terms containing $s$ and each $x \in X$ and closed under \emph{finite} Boolean combinations (again without $\wedge$) and subterms.  For a term $t$, we let $\floor{t} \in \@K(X)$ be its evaluation in $\@K(X)$.  Then $\floor{T} := \{\floor{t} \mid t \in T\} \setle \@K(X)$ is a countable ($\omega$-)Boolean subalgebra containing $a$ and the generators $X$.  Each $\sigma$-ultrafilter $U \setle \@K(X)$ restricts to an ($\omega$-)ultrafilter $U \cap \floor{T} \setle \floor{T}$, which determines $U$ since $X \setle \floor{T}$; conversely, it is easy to show by induction on terms $t \in T$ that an arbitrary ($\omega$-)ultrafilter $V \setle \floor{T}$ is the restriction of some $\sigma$-ultrafilter $U \setle \@K(X)$ (namely, the unique $U$ with the same restriction as $V$ to $X$) iff
\begin{itemize}
\item[($*$)]  for every $t = \bigvee_i t_i \in T$, if $\floor{t} \in V$, then $\floor{t_i} \in V$ for some $i$.
\end{itemize}
By Baire category, the set of all such ultrafilters $V \setle \floor{T}$ is a dense $G_\delta$ in the Stone space $\@S_\omega(\floor{T})$ of $\floor{T}$ (i.e., the compact Polish space of all ultrafilters on $\floor{T}$), since for each $t = \bigvee_i t_i$ as in ($*$), the basic clopen set $[\floor{t}] := \{V \in \@S_\omega(\floor{T}) \mid \floor{t} \in V\}$ is the closure of $\bigcup_i [\floor{t_i}]$ (whence the set of $V$ obeying ($*$) for that $t$ is a dense open set $\neg [\floor{t}] \cup \bigcup_i [\floor{t_i}]$).  Since $\top > a \in \floor{T}$, $\neg [a] \setle \@S_\omega(\floor{T})$ is nonempty clopen, hence contains some $V$ satisfying ($*$), which extends to $U \in \@S_\sigma(\@K(X))$ not containing $a$.
\end{proof}

Hence, the above equivalence $\sigma\!{Alg}^\op \cong \!{Bor}'$ restricts to an equivalence between $\sigma\!{Bool}_\sigma^\op$ and the full subcategory of $\!{Bor}$ consisting of those Borel spaces (isomorphic to ones) of the form $\@S_\sigma(A)$ for some countably presented Boolean $\sigma$-algebra $A$.  For a free algebra $\@K(X)$, $X$ countable, $\sigma$-ultrafilters $U \in \@S_\sigma(\@K(X))$ are determined by their restrictions to $X$, i.e., $\@S_\sigma(\@K(X)) \cong 2^X$, with Borel $\sigma$-algebra generated by the subbasic Borel sets $[x] := \{U \in 2^X \mid x \in U\}$ for $x \in X$, i.e., the standard Borel $\sigma$-algebra induced by the product topology.  Passing to a countably presented quotient $\@K(X)/r \cong \down r$ amounts to restricting to a Borel subspace $[r] \setle 2^X$.  Thus the spaces $\@S_\sigma(A)$ for $A \in \sigma\!{Bool}_\sigma$ are exactly the \defn{standard Borel spaces} (isomorphic copies of Borel subspaces of Cantor space $2^\#N$).  Letting $\!{SBor} \setle \!{Bor}$ denote the full subcategory of standard Borel spaces, we have

\begin{corollary}[Loomis--Sikorski duality]
The functors $\@B, \@S_\sigma$ restrict to an adjoint equivalence
\begin{align*}
\sigma\!{Bool}_\sigma^\op \cong \!{SBor}.
&\qed
\end{align*}
\end{corollary}

This reduces \cref{thm:intro-sbor-init} to \cref{thm:intro-kboolk-init} for $\kappa = \omega_1$.

\section{Almost universal $\@L_{\kappa\kappa}$-theories}
\label{sec:authy}

In this section, we introduce the infinitary first-order logic whose theories present $\kappa$-complete Boolean subextensive categories via the standard ``syntactic category'' construction from categorical logic.  See \cite{MR}, \cite[D1]{Jeleph} for general background on categorical logic and syntactic categories (also \cite{Cscc} for a treatment of the logic $\@L_{\omega_1\omega}$ similar to that given here).  We then give a presentation of $\kappa\!{Bool}_\kappa^\op$ via such a theory, thereby characterizing $\kappa$-continuous extensive functors $\kappa\!{Bool}_\kappa^\op -> \!C$ for arbitrary $\!C \in \kappa\&{LimBSExtCat}$, which we use to prove \cref{thm:intro-kboolk-init}.

The logic in question is the fragment of the full $\kappa$-ary first order logic $\@L_{\kappa\kappa}$ consisting of formulas which do not contain $\forall$ and only contain $\exists$ when the existential is already provably unique relative to the background theory.  We call such formulas \emph{almost quantifier-free}, and the corresponding theories (inductively built from sentences which are universally quantified almost quantifier-free formulas relative to the predecessor theory) \emph{almost universal}.  This ``provably unique $\exists$'' is well-known in categorical logic as an essential part of \emph{finite-limit logic} or \emph{Cartesian logic}, the fragment of finitary first-order logic $\@L_{\omega\omega}$ (due to Coste \cite{Cos}; see also \cite[D1.3.4]{Jeleph}) used to present finitely complete categories.  The combination with arbitrary quantifier-free Boolean formulas seems to be new (although it contains Johnstone's \cite{J79} \emph{disjunctive logic}, used to present extensive categories without Booleanness).  Because of this, and because of the somewhat delicate inductions involving the ``provably unique $\exists$'' which may not be well-known outside categorical logic, we will develop the syntactic category in some detail, even though all the ideas involved are essentially standard.

\subsection{Formulas and theories}

Recall from \cref{sec:cat} our standing assumption that $\!{Set}'_\kappa$ is the full subcategory of $\!{Set}$ consisting of all $\kappa$-ary subsets of a fixed set $\#U$ of size $\ge \kappa$.

Let $\@L$ be a \defn{(single-sorted)}%
\footnote{We will only need single-sorted theories for our purposes; the multi-sorted generalization is straightforward.}
\defn{$\kappa$-ary first-order relational language}, consisting of, for each $X \in \!{Set}'_\kappa$, a set $\@L(X)$ of \defn{$X$-ary relation symbols}.%
\footnote{We note that we are using ``$X$-ary'' with a different meaning than ``$\kappa$-ary'': the former means indexed by $X$, whereas the latter means indexed by a set of size $<\kappa$.}

The \defn{$\@L_{\kappa\kappa}$-formulas} are defined inductively as follows.  The variables of formulas will be elements of the sets $X \in \!{Set}'_\kappa$; we only consider formulas whose free variables all belong to some $X \in \!{Set}'_\kappa$.  We will keep track of free variables of formulas; for $X \in \!{Set}'_\kappa$, we write $\@L_{\kappa\kappa}(X)$ for the set of $\@L_{\kappa\kappa}$-formulas $\phi$ with free variables from $X$ ($X$ is called the \defn{context} of $\phi$).
\begin{itemize}
\item  For $X \in \!{Set}'_\kappa$, a $Y$-ary tuple of variables $\vec{x} \in X^Y$ where $Y \in \!{Set}'_\kappa$, and a $Y$-ary relational symbol $R \in \@L(Y)$, we have an atomic formula $R(\vec{x}) \in \@L_{\kappa\kappa}(X)$.
\item  For $X \in \!{Set}'_\kappa$ and $x, y \in X$, we have an atomic formula $(x = y) \in \@L_{\kappa\kappa}(X)$.
\item  For $X, I \in \!{Set}'_\kappa$ and $\phi, \phi_i \in \@L_{\kappa\kappa}(X)$ for each $i \in I$, we have $\neg \phi, \bigvee_{i \in I} \phi_i, \bigwedge_{i \in I} \phi_i \in \@L_{\kappa\kappa}(X)$ (when $I = \emptyset$, we write $\bot := \bigvee_{i \in \emptyset} \phi_i$ and $\top := \bigwedge_{i \in \emptyset} \phi_i$; when $I = n \in \#N$, we write $\phi_0 \vee \dotsb \vee \phi_{n-1} := \bigvee_{i \in n} \phi_i$, and similarly for $\wedge$).
\item  For $X \in \!{Set}'_\kappa$, $\phi \in \@L_{\kappa\kappa}(X)$, and $X = Y \cup Z$, we have $(\exists Y) \phi, (\forall Y) \phi \in \@L_{\kappa\kappa}(Z)$.%
\footnote{We allow $Y \cap Z \ne \emptyset$ so that a formula may always be regarded as having more free variables.}
\end{itemize}
We identify as usual two formulas $\phi, \psi \in \@L_{\kappa\kappa}(X)$ in the same context if they differ only in change of bound variables.
We denote \defn{variable substitution} by
\begin{equation*}
\@L_{\kappa\kappa}(Y) \ni \phi |--> [f]\phi \in \@L_{\kappa\kappa}(X)
\end{equation*}
for $f : Y -> X \in \!{Set}'_\kappa$; we automatically extend $f$ in the notation $[f]\phi$ by the identity map if necessary (so that e.g., $[x |-> y]\phi$ makes sense even if $\phi$ has free variables besides $x$).
An \defn{$\@L_{\kappa\kappa}$-theory} $\@T$ is a set of $\@L_{\kappa\kappa}$-sentences $\@T \setle \@L_{\kappa\kappa}(\emptyset)$.

We use a Gentzen sequent calculus-type proof system for $\@L_{\kappa\kappa}$.  We will not give the full details, for which see e.g., \cite[D1.3]{Jeleph}; however, we will mention some of the key features.  Sentences being proved are always universally quantified implications or \defn{sequents} $(\forall X) (\phi => \psi)$ (sometimes written $\phi |-_X \psi$) between two formulas $\phi, \psi \in \@L_{\kappa\kappa}(X)$ in the same context.  Since we are dealing with an infinitary logic, it is important to note that we do \emph{not} include any ``complete distributivity'' or ``axiom of choice''-type inference rules connecting $\bigwedge, \bigvee$ or $\bigwedge, \exists$ (as would be required to make the proof system complete with respect to the usual semantics in set-based models), \emph{except} for \eqref{AC} below.  The actual inference rules are the obvious infinitary generalizations of \cite[D1.3.1(a--h)]{Jeleph} (where the rules (f), (g) must allow quantifiers over multiple variables) and the law of excluded middle, together with the following axiom schema:%
\footnote{This expresses a version of the axiom of choice.  However, we will restrict below the existential quantifier $\exists$ so that the choice is always unique.}
\begin{align}
\phantomsection
\tag{AC}
\label{AC}
(\forall Z) \left(\bigwedge_{i \in I} (\exists Y_i) \phi_i => (\exists \bigcup_i Y_i) (\bigwedge_i \phi_i)\right)
\end{align}
for $X = \bigsqcup_i Y_i \sqcup Z \in \!{Set}'_\kappa$ (note the disjoint union) and $\phi_i \in \@L_{\kappa\kappa}(Y_i \sqcup Z)$.
We write as usual $\@T |- \sigma$ for a sequent $\sigma$ if it is provable from the theory $\@T$.
We also slightly abuse terminology by referring to provability of certain sentences which are not sequents, with the obvious meanings: a universally quantified bi-implication $(\forall X) (\phi <=> \psi)$ is provable iff both implications are; a general universally quantified formula $(\forall X) \phi$ is provable iff $(\forall X) (\top => \phi)$ is; etc.

Given a theory $\@T$, we inductively define the class of \defn{$\@T$-almost quantifier-free} (or \defn{$\@T$-aqf}) $\@L_{\kappa\kappa}$-formulas as follows, where by a \defn{$\@T$-aqf sequent} we mean one between two $\@T$-aqf formulas: a formula $\phi$ is $\@T$-aqf if $\phi$ does not contain $\forall$, and every existential subformula $(\exists Y) \psi$ occurring in the construction of $\phi$, where $\psi \in \@L_{\kappa\kappa}(X)$ with $X = Y \cup Z \in \!{Set}'_\kappa$, is \defn{$\@T$-provably unique}, meaning
\begin{align*}
\@T |- (\forall X \sqcup Y') (\psi \wedge [Y -> Y']\psi => \bigwedge_{y \in Y} (y = f(y)))
\end{align*}
where $Y \cong Y' \in \!{Set}'_\kappa$ is a copy of $Y$ disjoint from $X$,
and moreover this proof itself only involves $\@T$-aqf sequents.  A theory $\@T$ is itself \defn{almost universal} if there is a well-founded relation $\prec$ on $\@T$ such that each sentence in $\@T$ is an aqf sequent with respect to its $\prec$-predecessors.  In other words, $\@T$ can be constructed by repeatedly adding aqf sequents with respect to the preexisting theory.

We will only be considering aqf formulas and almost universal theories.  \emph{Thus, we henceforth restrict the proof system to proofs involving only aqf sequents; in particular, we drop the inference rules for $\forall$ \cite[D1.3.1(g)]{Jeleph}.  When we say $\@T$-provable, we mean using only $\@T$-aqf sequents.}

Let $\@L_{\kappa\kappa}^{\@T\aqf}(X) \setle \@L_{\kappa\kappa}(X)$ denote the subset of $\@T$-aqf formulas.  For $\phi \in \@L_{\kappa\kappa}^{\@T\aqf}(X)$, let $[\phi] = [\phi]_\@T$ denote its $\@T$-provable equivalence class, i.e.,
\begin{align*}
[\phi] = [\psi]  \coloniff  \@T |- (\forall X) (\phi <=> \psi)
\end{align*}
(using only $\@T$-aqf sequents).  Let
\begin{equation*}
\@L_{\kappa\kappa}^{\@T\aqf}(X)/\@T \in \kappa\!{Bool}
\end{equation*}
denote the \defn{Lindenbaum--Tarski algebra} of $\@T$-equivalence classes of $\@T$-aqf formulas, partially ordered by $\@T$-provable implication:
\begin{align*}
[\phi] \le [\psi] \coloniff  \@T |- (\forall X) (\phi => \psi),
\end{align*}
and with $\kappa$-Boolean operations given by the logical connectives (which are easily seen to satisfy the $\kappa$-Boolean algebra axioms mod $\@T$-provability).

As $X$ varies, these algebras are related via the variable substitution homomorphisms
\begin{align*}
[f] : \@L_{\kappa\kappa}^{\@T\aqf}(X)/\@T --> \@L_{\kappa\kappa}^{\@T\aqf}(Y)/\@T \in \kappa\!{Bool}
\end{align*}
for $f : X -> Y \in \!{Set}'_\kappa$.  We thus have a functor%
\footnote{By adding some operations to encode equality as well as provably unique $\exists$, we can enhance this functor to an ``aqf hyperdoctrine'' (see \cite{Law}), which can be used in place of the syntactic category below.  We have chosen the syntactic categories route, because the literature on syntactic categories seems to be better established.}
\begin{align*}
\@L_{\kappa\kappa}^{\@T\aqf}/\@T : \!{Set}'_\kappa --> \kappa\!{Bool}.
\end{align*}

\subsection{Syntactic categories and models}

Let $\@T$ be an almost universal $\@L_{\kappa\kappa}$-theory over a language $\@L$.  The \defn{syntactic category} $\ang{\@L \mid \@T}$ has:
\begin{itemize}
\item  objects: pairs $(X, \alpha)$ where $X \in \!{Set}'_\kappa$ and $\alpha \in \@L_{\kappa\kappa}^{\@T\aqf}(X)$;
\item  morphisms $(X, \alpha) -> (Y, \beta)$: $\@T$-provable equivalence classes $[\phi] \in \@L_{\kappa\kappa}^{\@T\aqf}(X \sqcup Y)/\@T$,%
\footnote{We may either fix a disjoint union for each $X, Y$ beforehand, or else consider indexed families of equivalence classes of formulas, one for each choice of disjoint union $X \sqcup Y$ equipped with injections $u : X `-> X \sqcup Y$ and $v : Y `-> X \sqcup Y$, which are compatible via substitution along the canonical bijections between different disjoint unions.  For simplicity, we will generally abuse notation and assume $X, Y$ are already disjoint, so that we may take $X \sqcup Y := X \cup Y$ (otherwise we would have to write e.g., $[u]\alpha \wedge [v]\beta$ instead of $\alpha \wedge \beta$ in $\sigma$).}
such that ``$\@T$ proves that $\phi$ is the graph of a function $\alpha -> \beta$'', i.e.,
\begin{alignat*}{3}
\@T &|- (\forall X \sqcup Y \sqcup Y') \Big(& \phi \wedge [Y -> Y']\phi &=> \alpha \wedge \beta \wedge \bigwedge_{Y \ni y |-> y' \in Y'} (y = y') & \Big) =: \sigma, \\
\@T \cup \{\sigma\} &|- (\forall X) \Big(& \alpha &=> (\exists Y) \phi & \Big) =: \tau,
\end{alignat*}
where $Y' \cong Y$ is a disjoint copy of $Y$ (note that $\tau$ is $(\@T \cup \{\sigma\})$-aqf).
\end{itemize}

\begin{proposition}
\label{thm:lkk-syncat}
$\ang{\@L \mid \@T}$ is a $\kappa$-complete Boolean $\kappa$-subextensive category, with:
\begin{itemize}
\item  identity $1_{(X, \alpha)} : (X, \alpha) -> (X, \alpha)$ given by
$[\alpha \wedge \bigwedge_{X \ni x |-> x' \in X'} (x = x')]$ (where $X'$ is a disjoint copy of $X$);
\item  composition of $[\phi] : (X, \alpha) -> (Y, \beta)$ and $[\psi] : (Y, \beta) -> (Z, \gamma)$ given by $[(\exists Y) (\phi \wedge \psi)]$;
\item  product of $(X_i, \alpha_i)$, for $i \in I \in \!{Set}'_\kappa$, given by $(\bigsqcup_i X_i, \bigwedge_i \alpha_i)$;
\item  equalizer of $[\phi], [\psi] : (X, \alpha) -> (Y, \beta)$ given by $(X, (\exists Y) (\phi \wedge \psi))$;
\item  a morphism $[\phi] : (X, \alpha) -> (Y, \beta)$ monic iff $(\exists X) \phi$ is $\@T$-provably unique;
\item  subobjects of $(X, \alpha)$ given by
\begin{align*}
\Sub_{\ang{\@L \mid \@T}}(X, \alpha) &\cong \down [\alpha] = \{[\phi] \mid \@T |- (\forall X) (\phi => \alpha)\} \setle \@L_{\kappa\kappa}^{\@T\aqf}(X)/\@T \\
(1_{(X, \phi)} : (X, \phi) `-> (X, \alpha)) &<-| [\phi] \\
([\psi] : (Y, \beta) `-> (X, \alpha)) &|-> [(\exists Y) \psi];
\end{align*}
\item  join of a $\kappa$-ary family of $(X, \phi_i) `-> (X, \alpha)$ given by $(X, \bigvee_i \phi_i) `-> (X, \alpha)$;
\item  nonempty meet of $(X, \phi_i) `-> (X, \alpha)$ given by $(X, \bigwedge_i \phi_i) `-> (X, \alpha)$;
\item  complement of $(X, \phi) `-> (X, \alpha)$ given by $(X, \alpha \wedge \neg \phi) `-> (X, \alpha)$;
\item  pullback of a subobject $(X, \phi) `-> (X, \top)$ corresponding to $[\phi] \in \@L_{\kappa\kappa}^{\@T\aqf}(X)/\@T$ along a morphism of the form $(\pi_{f(x)})_{x \in X} : (Y, \top) \cong (1, \top)^Y -> (1, \top)^X \cong (X, \top)$ for some $f : X -> Y$, where $\pi_y : \prod_{y \in Y} (1, \top) -> (1, \top)$ is the $y$th projection, given by the variable substitution $[f] : \@L_{\kappa\kappa}^{\@T\aqf}(X)/\@T -> \@L_{\kappa\kappa}^{\@T\aqf}(Y)/\@T$.
\end{itemize}
\end{proposition}
\begin{proof}
By straightforward constructions of explicit proofs witnessing the necessary conditions; see \cite[D1.4]{Jeleph}, \cite[8.2.1]{MR} for details (in the finitary case; the $\kappa$-ary case is analogous).  We only comment on the role of \eqref{AC}, which is used to check the universal property of the product $\prod_i (X_i, \alpha_i) = (\bigsqcup_i X_i, \bigwedge_i \alpha_i)$: given a cone $([\phi_i] : (Y, \beta) -> (X_i, \alpha_i))_i$, the induced morphism $(Y, \beta) -> (\bigsqcup_i X_i, \bigwedge_i \alpha_i)$ is given by $[\bigwedge_i \phi_i]$; to check the totality condition in the definition of morphism, we need $\beta => (\exists \bigsqcup_i X_i) (\bigwedge_i \phi_i)$, which follows from $\beta => \bigwedge_i (\exists X_i) \phi_i$ (because each $[\phi_i]$ is a morphism) and $\bigwedge_i (\exists X_i) \phi_i => (\exists \bigsqcup_i X_i) (\bigwedge_i \phi_i)$ by \eqref{AC}.
\end{proof}

We have a distinguished object $H := (1, \top) \in \ang{\@L \mid \@T}$, whose $\kappa$-ary powers are by \cref{thm:lkk-syncat} given by $H^X = (X, \bigwedge_i \top) \cong (X, \top)$ for each $X \in \!{Set}'_\kappa$; and each relation symbol $R \in \@L(X)$ yields a subobject $R^\@H := (X, R(1_X)) `-> (X, \top) = H^X \in \ang{\@L \mid \@T}$.
We next show that $\ang{\@L \mid \@T}$ is the $\kappa$-complete Boolean $\kappa$-subextensive category ``freely generated by an object $H$, together with subobjects $R^\@H \setle H^X$ for each $R \in \@L(X)$, satisfying the relations in $\@T$''.  This requires defining the category of all such data in an arbitrary $\kappa$-complete Boolean $\kappa$-subextensive category $\!C$, i.e., the notion of a model of $\@T$ in $\!C$; see \cite[D1.2]{Jeleph}, \cite[\S2.3]{MR} for the finitary case.

Let $\@L$ be a language and $\!C$ be a $\kappa$-complete Boolean $\kappa$-subextensive category.  An \defn{$\@L$-structure $\@M = (M, R^\@M)_{R \in \@L}$ in $\!C$} consists of:
\begin{itemize}
\item  an \defn{underlying object} $M \in \!C$;
\item  for each relation symbol $R \in \@L(X)$, a subobject $R^\@M \setle \prod M^X$, called the \defn{interpretation} of $R$ in $\@M$.
\end{itemize}
Given an $\@L$-structure $\@M$ in $\!C$, we inductively define for certain $\forall$-free $\@L_{\kappa\kappa}$-formulas $\phi \in \@L_{\kappa\kappa}(X)$ an \defn{interpretation} $\phi^\@M \setle H^X$ of $\phi$ in $\@M$, as follows; we say that $\phi$ is \defn{interpretable} in $\@M$ if it has an interpretation.
\begin{itemize}
\item  An atomic formula $R(\vec{x}) \in \@L_{\kappa\kappa}(X)$, where $\vec{x} \in X^Y$ and $R \in \@L(Y)$, is interpreted as the pullback of $R^\@M \setle M^Y$ along $(\pi_{x_y})_{y \in Y} : M^X -> M^Y$ (where $\pi_x : M^X -> M$ is the $x$th projection).
\item  An atomic formula $(x = y) \in \@L_{\kappa\kappa}(X)$, where $x, y \in X$, is interpreted as the equalizer of the projections $\pi_x, \pi_y : M^X -> M$.
\item  If $\phi, \phi_i \in \@L_{\kappa\kappa}(X)$ ($I \in \!{Set}'_\kappa$) are interpretable in $\@M$, then so are $\neg \phi, \bigvee_i \phi_i, \bigwedge_i \phi_i \in \@L_{\kappa\kappa}(X)$, given by applying the corresponding $\kappa$-Boolean operations in $\Sub_\!C(M^X)$.
\item  If $\phi \in \@L_{\kappa\kappa}(X)$ is interpretable, where $X = Y \cup Z$, and the composite
\begin{align*}
\phi^\@M `-> M^X --->{\pi_{X \setminus Y}} M^{X \setminus Y}
\end{align*}
(where $\pi_{X \setminus Y}$ is the projection) is monic, then $(\exists Y) \phi \in \@L_{\kappa\kappa}(Z)$ is interpreted as the pullback of this composite along the projection $M^Z -> M^{X \setminus Y}$.
\end{itemize}
A sentence $\sigma = (\forall X) \phi$ with $\phi$ $\forall$-free is \defn{interpretable} in $\@M$ if $\phi$ is, and \defn{satisfied} in $\@M$, written $\@M |= \sigma$, if $\phi^\@M = M^X$ (thus if $\sigma = (\forall X) (\phi => \psi)$ is a sequent, this is equivalent to $\phi^\@M \setle \psi^\@M \setle M^X$).  For a theory $\@T$, $\@M$ is a \defn{model} of $\@T$, written $\@M |= \@T$, if every $\sigma \in \@T$ is satisfied in $\@M$.

\begin{proposition}[soundness]
If $\@M |= \@T$, then every $\@T$-aqf formula is interpretable in $\@M$, and every $\@T$-aqf sequent $(\forall X) (\phi => \psi)$ proved by $\@T$ is satisfied in $\@M$.
\end{proposition}
\begin{proof}
First, one proves by an easy induction that interpretation of formulas in $\@M$ is preserved under variable substitutions: for $f : Y -> X \in \!{Set}'_\kappa$, if $\phi \in \@L_{\kappa\kappa}(Y)$ is interpretable, then so is $[f]\phi \in \@L_{\kappa\kappa}(X)$, given by the pullback of $\phi^\@M \setle M^Y$ along $(\pi_{f(y)})_{y \in Y} : M^X -> M^Y$.

Next, one proves by induction that if a sequent interpretable in $\@M$ (but not yet known to be $\@T$-aqf) is proved by $\@T$ using only other sequents interpretable in $\@M$, then it must be satisfied in $\@M$.  This is again straightforward; see \cite[D1.3.2]{Jeleph}.  Again, we only verify \eqref{AC}.  Let $X = \bigsqcup_i Y_i \sqcup Z \in \!{Set}'_\kappa$ and $\phi_i \in \@L_{\kappa\kappa}(Y_i \sqcup Z)$; we must show that
\begin{align*}
(\bigwedge_i (\exists Y_i) \phi_i)^\@M \setle ((\exists \bigcup_i Y_i) (\bigwedge_i \phi_i))^\@M \setle M^Z,
\end{align*}
assuming that the left (and right) interpretation is defined.  Interpretability of $(\exists Y_i) \phi_i \in \@L_{\kappa\kappa}(Z)$ means that $\phi_i \in \@L_{\kappa\kappa}(Y_i \sqcup Z)$ is interpretable, and the composite
\begin{align*}
\phi_i^\@M `-> M^{Y_i \sqcup Z} -> M^Z
\end{align*}
(the second map is the projection) is monic; this composite is then $((\exists Y_i) \phi_i)^\@M \setle M^Z$.  The wide pullback of all these composites is then $(\bigwedge_i (\exists Y_i) \phi_i)^\@M \setle M^Z$.  By standard facts about pullbacks (e.g., do the calculation in $\!{Set}$ using Yoneda), this wide pullback is the composite of the wide pullback of the pullbacks of $\phi_i^\@M `-> M^{Y_i \sqcup Z}$ along the projections $M^X -> M^{Y_i \sqcup Z}$, with the projection $M^X -> M^Z$, which is exactly $((\exists \bigcup_i Y_i) (\bigwedge_i \phi_i))^\@M$ (this uses the first part of the proof to relate the interpretations of $\phi_i \in \@L_{\kappa\kappa}(Y_i \sqcup Z)$ and $\phi_i \in \@L_{\kappa\kappa}(X)$).

Finally, one proves that every $\@T$-aqf formula is interpretable in $\@M$, by induction on the definition of $\@T$-aqf.  This uses a nested induction on the formula in question, using that the ``$\@T$-provably unique'' condition in the definition of $\@T$-aqf (which is satisfied in $\@M$ by the previous part of the proof and the outer induction hypothesis) exactly translates to the condition for an existential quantifier to be interpretable.
\end{proof}

\begin{corollary}
If $\@T$ is almost universal, then either $\@M |= \@T$, or there is a sequent in $\@T$ which is interpretable but not satisfied in $\@M$.
\end{corollary}
\begin{proof}
By induction on the well-founded relation witnessing that $\@T$ is almost universal.
\end{proof}

Let $\@T$ be an almost universal theory.  For two models $\@M, \@N \models \@T$ in $\!C$, a \defn{$\@T$-aqf embedding} $f : \@M -> \@N$ is a morphism $f : M -> N \in \!C$ such that for every $\@T$-aqf formula $\phi \in \@L_{\kappa\kappa}^{\@T\aqf}(X)$, the induced morphism $f^X : M^X -> N^X$ restricts to a (necessarily unique) morphism $\phi^\@M -> \phi^\@N$:%
\footnote{By considering the formulas $\neg \phi$, one sees that each such square must in fact be a pullback, i.e., ``$\phi^\@M = f^{-1}(\phi^\@N)$''.  By considering the formulas $x \ne y$, one sees that $f$ must be monic.  We will not need these facts.}
\begin{equation*}
\begin{tikzcd}
\phi^\@M \dar[hook] \rar[dashed] & \phi^\@N \dar[hook] \\
M^X \rar["f^X"] & N^X
\end{tikzcd}
\end{equation*}
Let $\!{\@T-Mod}(\!C)$ be the category of models of $\@T$ in $\!C$ and $\@T$-aqf embeddings.

Since the definition of $\@L$-structure in $\!C$ and the interpretation of $\@T$-aqf formulas only use the $\kappa$-complete Boolean $\kappa$-subextensive structure of $\!C$, we may pushforward a model $\@M \in \!{\@T-Mod}(\!C)$ across a $\kappa$-continuous $\kappa$-subextensive functor $\@F : \!C -> \!D$, for another $\kappa$-complete Boolean $\kappa$-subextensive category $\!D$, to obtain a model
\begin{align*}
\!{\@T-Mod}(\@F)(\@M) := \@F(\@M) := (\@F(M), \@F(R^\@M))_{R \in \@L} \in \!{\@T-Mod}(\!D),
\end{align*}
which has the property that $\@F(\phi^\@M) = \phi^{\@F(\@M)} \setle \@F(M)^X \cong \@F(M^X)$ for any $\phi \in \@L_{\kappa\kappa}^{\@T\aqf}(X)$.  Thus, for every $\@T$-aqf embedding $f : \@M -> \@N \in \!{\@T-Mod}(\!C)$, its $\@F$-image $\@F(f) : \@F(M) -> \@F(N) \in \!D$ is a $\@T$-aqf embedding $\@F(\@M) -> \@F(N) \in \!{\@T-Mod}(\!D)$.  So $\@F : \!C -> \!D \in \kappa\&{LimBSExtCat}$ induces a functor
\begin{align*}
\!{\@T-Mod}(\@F) : \!{\@T-Mod}(\!C) -> \!{\@T-Mod}(\!D) \in \&{Cat}.
\end{align*}
For a natural transformation $\xi : \@F -> \@G : \!C -> \!D \in \kappa\&{LimBSExtCat}$, for any model $\@M \in \!{\@T-Mod}(\!C)$, the component $\xi_M : \@F(M) -> \@G(M)$ is a $\@T$-aqf embedding $\@F(\@M) -> \@G(\@M)$, since for any $\phi \in \@L_{\kappa\kappa}^{\@T\aqf}(X)$, by naturality of $\xi$ we have a commutative square
\begin{equation*}
\begin{tikzcd}[column sep=4em]
\@F(\phi^\@M) = \phi^{\@F(\@M)} \dar[hook] \rar["\xi_{\phi^\@M}"] & \@G(\phi^\@M) = \phi^{\@G(\@M)} \dar[hook] \\
\@F(M^X) = \@F(M)^X \rar["\xi_{M^X} = \xi_M^X"] & \@G(M^X) = \@G(M)^X;
\end{tikzcd}
\end{equation*}
thus $\xi$ lifts to a natural transformation
\begin{align*}
\!{\@T-Mod}(\xi) : \!{\@T-Mod}(\@F) -> \!{\@T-Mod}(\@G) : \!{\@T-Mod}(\!C) -> \!{\@T-Mod}(\!D)
\end{align*}
with components $\!{\@T-Mod}(\xi)_\@M := \xi_\@M := \xi_M$ for $\@M \in \!{\@T-Mod}(\!C)$.  So we have a (strict) 2-functor
\begin{align*}
\!{\@T-Mod} : \kappa\&{LimBSExtCat} &--> \&{Cat}.
\end{align*}

Returning to the syntactic category $\ang{\@L \mid \@T}$, the object $H = (1, \top) \in \ang{\@L \mid \@T}$ together with the subobjects $R^\@H = (X, R(1_X)) \setle H^X$ form an $\@L$-structure $\@H = \@H_\@T$ in $\ang{\@L \mid \@T}$, for which it is easily seen by induction that $\phi^\@H = (X, \phi) \setle H^X$ for any $\phi \in \@L_{\kappa\kappa}^{\@T\aqf}(X)$,
whence $\@H$ is a model of $\@T$, called the \defn{universal model}.  We may finally state the universal property of $\ang{\@L \mid \@T}$:

\begin{proposition}
\label{thm:lkk-syncat-univ}
For any other $\kappa$-complete Boolean $\kappa$-subextensive category $\!C$, we have an equivalence of categories
\begin{align*}
\kappa\&{LimBSExtCat}(\ang{\@L \mid \@T}, \!C) &--> \!{\@T-Mod}(\!C) \\
\@F &|--> \@F(\@H) \\
(\xi : \@F -> \@G) &|--> \xi_\@H.
\end{align*}
(In other words, $\@H \in \!{\@T-Mod}(\ang{\@L \mid \@T})$ is a (non-strict) representation of the 2-presheaf $\!{\@T-Mod}$.)
\end{proposition}
\begin{proof}
This is again standard; see \cite[D1.4.7]{Jeleph}, \cite[8.2.4]{MR}.  The above functor is clearly faithful.
For fullness, given $\kappa$-continuous $\kappa$-subextensive $\@F, \@G : \ang{\@L \mid \@T} -> \!C \in \kappa\&{LimBSExtCat}$ and a $\@T$-aqf embedding $f : \@F(\@H) -> \@G(\@H)$, for each $\alpha \in \@L_{\kappa\kappa}^{\@T\aqf}(X)$, define $\xi_{(X, \alpha)} : \@F(X, \alpha) = \@F(\alpha^\@H) = \alpha^{\@F(\@H)} -> \alpha^{\@G(\@H)} = \@G(\alpha^\@M) = \@G(X, \alpha)$ to be the restriction of $f^X : \@F(H)^X -> \@G(H)^X$ (using that $f$ is a $\@T$-aqf embedding).
These form the components of a natural transformation $\xi : \@F -> \@G$: naturality with respect to the product projections $H^X = (X, \top) -> (1, \top) = H$ and the canonical inclusions $(X, \alpha) `-> (X, \top) = H^X$ is by definition; this implies naturality with respect to an arbitrary $[\phi] : (X, \alpha) -> (Y, \beta) \in \ang{\@L \mid \@T}$, by considering the factorization of $[\phi]$ through its graph $(X \sqcup Y, \phi)$:
\begin{equation*}
\begin{tikzcd}
(X, \alpha) \dar[hook] \ar[rr, bend left, "{[\phi]}"] & (X \sqcup Y, \phi) \dar[hook] \lar["\pi_X\vert\phi", "\cong"'] \rar["\pi_Y\vert\phi"'] & (Y, \beta) \dar[hook] \\
(X, \top) & (X \sqcup Y, \top) \lar["\pi_X"] \rar["\pi_Y"'] & (Y, \top)
\end{tikzcd}
\end{equation*}
For essential surjectivity, given a model $\@M \in \!{\@T-Mod}(\!C)$, define $\@F : \ang{\@L \mid \@T} -> \!C$ by sending each object $(X, \alpha)$ to (the domain of any monomorphism representative of the subobject) $\alpha^\@M \subseteq M^X$, and each morphism $[\phi] : (X, \alpha) -> (Y, \beta)$ to the unique $\alpha^\@M -> \beta^\@M \in \!C$ whose graph is $\phi^\@M \subseteq \alpha^\@M \times \beta^\@M \subseteq M^X \times M^Y = M^{X \sqcup Y}$; it is straightforward to check, using the explicit constructions in \cref{thm:lkk-syncat}, that $\@F$ is a $\kappa$-continuous $\kappa$-subextensive functor such that $\@F(\@H) \cong \@M$.
\end{proof}

\subsection{The theory of 2}

Let $\@L_2$ be the language with a single unary relation symbol $R_\top \in \@L_2(1)$.  Let $\@T_2$ be the $(\@L_2)_{\kappa\kappa}$-theory consisting of the sequents (where $x, y$ are distinct variables)
\begin{alignat*}{3}
&(\forall \{x, y\}) (& R_\top(x) \wedge R_\top(y) &=> (x = y) &), \\
&(\forall \{x, y\}) (& \neg R_\top(x) \wedge \neg R_\top(y) &=> (x = y) &), \\
&(\forall \emptyset) (& \top &=> (\exists x) R_\top(x) &), \\
&(\forall \emptyset) (& \top &=> (\exists x) \neg R_\top(x) &).
\end{alignat*}
The last two sequents are aqf relative to the first two, whence $\@T_2$ is almost universal.  It is easily seen that a model $\@M$ of $\@T_2$ in $\!C \in \kappa\&{LimBSExtCat}$ consists of an object $M \in \!C$ together with a subobject $R_\top^\@M \setle M$ which is a terminal object and whose complement is also a terminal object; in other words,

\begin{proposition}
For any $\kappa$-complete Boolean $\kappa$-subextensive category $\!C$,
$\!{\@T_2-Mod}(\!C)$ is the category of binary coproducts $1 \sqcup 1$ (equipped with cocone $u, v : 1 -> 1 \sqcup 1$) of the terminal object $1 \in \!C$.  \qed
\end{proposition}

\begin{corollary}
\label{thm:extcat-t2mod-unique}
If $\!C$ is extensive, then $\!{\@T_2-Mod}(\!C)$ is equivalent to the terminal category.  \qed
\end{corollary}

Now take $\!C := \kappa\!{Bool}_\kappa^\op$ (which is $\kappa$-complete Boolean $\kappa$-extensive by \cref{thm:kboolk-klimbext}).  By \cref{thm:kboolk-quotient}, we may identify subobjects of $A \in \kappa\!{Bool}_\kappa^\op$ with their elements $a \in A$.  The binary coproduct of the terminal object in $\kappa\!{Bool}_\kappa^\op$ is the binary product of the initial $\kappa$-Boolean algebra $2 = \{\bot, \top\} \in \kappa\!{Bool}_\kappa$, which is the free algebra $\@K(1) = \{\bot, 0, \neg 0, \top\}$ on one generator $0 \in \@K(1)$; the two product projections $\@K(1) ->> 2 \in \kappa\!{Bool}_\kappa$ take $0$ to $\top$ and $\bot$ respectively, which, when seen as subobjects of $\@K(1)$ in $\kappa\!{Bool}_\kappa^\op$, correspond via \cref{thm:kboolk-quotient} to the elements $0, \neg 0 \in \@K(1)$.  Thus, the model $\@M_2 \in \!{\@T_2-Mod}(\kappa\!{Bool}_\kappa^\op)$ given by \cref{thm:extcat-t2mod-unique} may be taken as (again using \cref{thm:kboolk-quotient})
\begin{align*}
\@M_2 := (\@K(1), 0) \in \!{\@T_2-Mod}(\kappa\!{Bool}_\kappa^\op).
\end{align*}

\begin{theorem}[$\@T_2$ presents $\kappa\!{Bool}_\kappa^\op$]
\label{thm:kboolk=syncat}
The $\kappa$-continuous subextensive functor $\@F_2 : \ang{\@L_2 \mid \@T_2} -> \kappa\!{Bool}_\kappa^\op$ corresponding to $\@M_2$ via \cref{thm:lkk-syncat-univ} is an equivalence of categories.
\end{theorem}
\begin{proof}
Since $\@F_2$ takes the universal model $\@H = ((1, \top), (1, R_\top(0)))$ to an isomorphic copy of $\@M_2 = (\@K(1), 0)$, we may assume that $\@F_2(1, \top) = \@K(1)$ and $\@F_2(1, R_\top(0)) = 0 \in \@K(1)$ (again identifying subobjects in $\kappa\!{Bool}_\kappa^\op$ with elements using \cref{thm:kboolk-quotient}).  Since $\@F_2$ preserves $\kappa$-ary products, we may further assume that $\@F_2(X, \top) = \@F_2((1, \top)^X) = \@K(1)^X = \@K(X) \in \kappa\!{Bool}_\kappa^\op$ for each $X \in \!{Set}'_\kappa$, and that $\@F_2$ takes the product projections $\pi_x : (X, \top) -> (1, \top) \in \ang{\@L \mid \@T}$ to the corresponding projections in $\kappa\!{Bool}_\kappa^\op$, i.e., the coproduct injections $\@K(x) : \@K(1) -> \@K(X) \in \kappa\!{Bool}_\kappa$ taking the generator $0 \in \@K(1)$ to $x \in \@K(X)$.


Let us say that a formula $\phi \in (\@L_2)_{\kappa\kappa}(X)$ is \defn{quantifier-and-equality-free} (\defn{qef}) if it contains neither quantifiers nor $=$.  In particular, such a formula is $\@T_2$-aqf (indeed $\emptyset$-aqf).  Note that such a formula is essentially a $\kappa$-Boolean algebra term in the generators $X$, except with each generator $x \in X$ replaced by the atomic formula $R_\top(x)$.  For each such atomic formula, the corresponding subobject $(X, R_\top(x)) \setle (X, \top) \in \ang{\@L_2 \mid \@T_2}$ is (by \cref{thm:lkk-syncat}) the pullback along the projection $\pi_x : (X, \top) -> (1, \top)$ of $(1, R_\top(0)) \setle (1, \top)$; since $\@F_2$ preserves pullbacks, $\@F_2(X, R_\top(x)) \setle \@F_2(X, \top) = \@K(X) \in \kappa\!{Bool}_\kappa^\op$ must be the pushout in $\kappa\!{Bool}_\kappa$ of the quotient $\@K(1) ->> \@K(1)/0$ along $\iota_x : \@K(1) -> \@K(X)$, which is the quotient $\@K(X) ->> \@K(X)/x$, i.e., the subobject (via \cref{thm:kboolk-quotient}) $x \in \@K(X)$.  So for each $(X, \top) \in \ang{\@L_2 \mid \@T_2}$, the $\kappa$-Boolean homomorphism
\begin{align*}
r_X := \@F_2|\Sub(X, \top) : (\@L_2)_{\kappa\kappa}^{\@T_2\aqf}(X)/\@T_2 \cong \Sub_{\ang{\@L_2 \mid \@T_2}}(X, \top) &-->> \Sub_{\kappa\!{Bool}_\kappa^\op}(\@K(X)) \cong \@K(X) \\
\intertext{(the first $\cong$ by \cref{thm:lkk-syncat}) maps}
(\@L_2)_{\kappa\kappa}^{\@T_2\aqf}(X)/\@T_2 \ni [R_\top(x)] &|--> x \in \@K(X).
\end{align*}
By freeness of $\@K(X)$, $r_X$ has a section
\begin{align*}
s_X : \@K(X) &`--> (\@L_2)_{\kappa\kappa}^{\@T_2\aqf}(X)/\@T_2 \\
X \ni x &|--> [R_\top(x)],
\end{align*}
whose image clearly consists of the equivalence classes of qef formulas.  Also, clearly, the $s_X$ form a natural transformation $\@K -> (\@L_2)_{\kappa\kappa}^{\@T_2\aqf}/\@T_2 : \!{Set}'_\kappa -> \kappa\!{Bool}$, i.e., for any $f : X -> Y \in \!{Set}'_\kappa$,
\begin{align*}
\tag{$*$}
s_Y(f_*(a)) = [f]s_X(a)  \qquad\text{for any $a \in \@K(X)$}
\end{align*}
(this clearly holds on generators $a = x \in X$); it follows that
\begin{align*}
\tag{$\dagger$}
r_Y([f][\phi]) = f_*(r_X[\phi])  \qquad\text{for qef $\phi \in (\@L_2)_{\kappa\kappa}(X)$}
\end{align*}
(since $[\phi] \in \im(s_X)$, so $r_Y([f][\phi]) = r_Y([f]s_X(r_X[\phi])) = r_Y(s_Y(f_*(r_X[\phi]))) = f_*(r_X[\phi])$).

\begin{lemma}[quantifier and equality elimination for $\@T_2$]
\label{lm:qee}
Every $\phi \in (\@L_2)_{\kappa\kappa}^{\@T_2\aqf}(X)$ is $\@T_2$-provably equivalent to a qef formula.  Thus, the homomorphisms $r_X, s_X$ above are inverse isomorphisms.
\end{lemma}
\begin{proof}
An equality $(x = y)$ is easily $\@T_2$-equivalent to the qef formula $R_\top(x) <-> R_\top(y)$.  Thus by induction, it is enough to prove that for any qef $\phi \in (\@L_2)_{\kappa\kappa}(X)$ and $X = Y \sqcup Z$, if the existential $(\exists Y) \phi \in (\@L_2)_{\kappa\kappa}(Z)$ is $\@T_2$-provably unique, then it is $\@T_2$-equivalent to a qef formula.

That $(\exists Y) \phi$ is $\@T_2$-provably unique means
\begin{align*}
\@T_2 |- (\forall X \sqcup Y') \left([f]\phi \wedge [g]\phi => \bigwedge_{y \in Y} (f(y) = g(y))\right),
\end{align*}
where $Y'$ is a copy of $Y$ disjoint from $X$, with bijection $g : Y \cong Y'$ extending to $g : X = Y \sqcup Z \cong Y' \sqcup Z \setle X \sqcup Y'$, and $f : X `-> X \sqcup Y'$ is the inclusion.  Eliminating the equalities $f(y) = g(y)$ as above yields
\begin{align*}
\@T_2 |- (\forall X \sqcup Y') \left([f]\phi \wedge [g]\phi => \bigwedge_{y \in Y} (R_\top(f(y)) <-> R_\top(g(y)))\right),
\end{align*}
i.e.,
\begin{align*}
[f][\phi] \wedge [g][\phi] \le \bigwedge_{y \in Y} ([R_\top(f(y))] <-> [R_\top(g(y))]) \in (\@L_2)_{\kappa\kappa}^{\@T_2\aqf}(X \sqcup Y')/\@T_2.
\end{align*}
Applying $r_{X \sqcup Y'}$ and ($\dagger$) yields
\begin{align*}
f_*(r_X[\phi]) \wedge g_*(r_X[\phi]) \le \bigwedge_{y \in Y} (f(y) <-> g(y)) \in \@K(X \sqcup Y').
\end{align*}
By \cref{thm:lusin-suslin}, there is a retraction $h : \@K(X) ->> \@K(Z) \in \kappa\!{Bool}_\kappa$ of $i_*$, where $i : Z `-> X$ is the inclusion, such that
\begin{align*}
r_X[\phi] = i_*(h(r_X[\phi])) \wedge \bigwedge_{y \in Y} (i_*(h(y)) <-> y) \in \@K(X).
\end{align*}
Applying $s_X$ and ($*$) yields
\begin{align*}
[\phi] = [i]s_X(h(r_X[\phi])) \wedge \bigwedge_{y \in Y} ([i]s_X(h(y)) <-> [R_\top(y)]) \in (\@L_2)_{\kappa\kappa}^{\@T_2\aqf}(X)/\@T_2.
\end{align*}
So letting $s_X(h(r_X[\phi])) = [\psi]$ and $s_X(h(y)) = [\psi_y]$ for some qef $\psi, \psi_y \in (\@L_2)_{\kappa\kappa}(Z)$, we have
\begin{align*}
\@T_2 |- (\forall X) \left(\phi <=> \psi \wedge \bigwedge_{y \in Y} (\psi_i <-> R_\top(y))\right).
\end{align*}
This easily yields
\begin{align*}
\@T_2 |- (\forall Z) \left((\exists Y) \phi <=> \psi \wedge (\exists Y) \bigwedge_{y \in Y} (\psi_i <-> R_\top(y))\right),
\end{align*}
where $(\exists Y) \bigwedge_{y \in Y} (\psi_i <-> R_\top(y))$ is $\@T_2$-provably unique by the first two axioms of $\@T_2$; but then the last two axioms of $\@T_2$ easily prove $\bigwedge_{y \in Y} (\exists y) (\psi_i <-> R_\top(y))$, which implies $(\exists Y) \bigwedge_{y \in Y} (\psi_i <-> R_\top(y))$ by \eqref{AC}.  So $\@T_2 |- (\forall Z) ((\exists Y) \phi <=> \psi)$.
\end{proof}

So $\@F_2 : \ang{\@L_2 \mid \@T_2} -> \kappa\!{Bool}_\kappa^\op$ is an isomorphism on each subobject poset $\Sub_{\ang{\@L_2 \mid \@T_2}}(X, \top)$.  It follows easily that it restricts to an isomorphism on every subobject poset $\Sub_{\ang{\@L_2 \mid \@T_2}}(X, \alpha)$, where we may assume $\alpha$ is qef by \cref{lm:qee}, by considering the commutative diagram
\begin{equation*}
\begin{tikzcd}
\Sub_{\ang{\@L_2 \mid \@T_2}}(X, \alpha) \dar[hook] \rar["\@F_2"] & \Sub_{\kappa\!{Bool}_\kappa^\op}(\@K(X)/r_X[\alpha]) \dar[hook] \rar[phantom, "\cong"] &[-1em] \down r_X[\alpha] \dar[hook] \\
\Sub_{\ang{\@L_2 \mid \@T_2}}(X, \top) \rar["\@F_2"'] & \Sub_{\kappa\!{Bool}_\kappa^\op}(\@K(X)) \rar[phantom, "\cong"] & \@K(X)
\end{tikzcd}
\end{equation*}
where the vertical arrows are inclusions (note: not $\kappa$-Boolean homomorphisms): injectivity of $\@F_2|\Sub(X, \alpha)$ follows immediately from that of $\@F_2|\Sub(X, \top)$; for surjectivity, for $b \in \down r_X[\alpha]$, letting $\beta \in (\@L_2)_{\kappa\kappa}(X)$ be qef with $[\beta] = s_X(b) \le s_X(r_X[\alpha]) = [\alpha]$, we have $(X, \beta) \in \Sub(X, \alpha)$ with $\@F_2(X, \beta) = r_X[\beta] = r_X(s_X(b)) = b$.

Finally, every $A \in \kappa\!{Bool}_\kappa$ is (up to isomorphism) a quotient of a free algebra $\@K(X)$ for some $X \in \!{Set}'_\kappa$, i.e., a subobject in $\kappa\!{Bool}_\kappa^\op$ of some $\@K(X) = \@F_2(X, \top)$; so $\@F_2$ is essentially surjective.  So $\@F_2$ is a finitely continuous functor between finitely complete categories which is a bijection on each subobject poset (in the usual terminology, \defn{conservative} and \defn{full on subobjects}) as well as essentially surjective, hence an equivalence (see e.g., \cite[D3.5.6]{Jeleph}).
\end{proof}

\begin{theorem}
\label{thm:kboolk-init}
$\kappa\!{Bool}_\kappa^\op$ is an initial object in $\kappa\&{LimBExtCat}$.
\end{theorem}
\begin{proof}
By \cref{thm:kboolk-klimbext}, \cref{thm:kboolk=syncat}, \cref{thm:lkk-syncat-univ}, and \cref{thm:extcat-t2mod-unique}.
\end{proof}

\bigskip\noindent
Department of Mathematics \\
University of Michigan \\
Ann Arbor, MI 48109, USA \\
Email: \nolinkurl{ruiyuan@umich.edu}

\end{document}